\documentclass[11pt,reqno]{amsart}
\usepackage{color, amsmath,amssymb, amsfonts, amstext,amsthm, latexsym}
\allowdisplaybreaks
\numberwithin{equation}{section}
\newtheorem{theorem}{Theorem}[section]
\newtheorem{definition}[theorem]{Definition}
\newtheorem{lemma}[theorem]{Lemma}
\newtheorem{remark}[theorem]{Remark}
\newtheorem{proposition}[theorem]{Proposition}
\newtheorem{coro}[theorem]{Corollary}

\begin{document}

\title[On Stochastic Shell Models of Turbulence] 
{On Stochastic Shell Models of Turbulence} 



\author[C. Wijeratne] 
{Chandana Wijeratne} 
\address{
C. Wijeratne, Sri Lanka Institute of Information Technology,
New Kandy Road,
Malabe,
10115,
Sri Lanka}
\email{chandana.w@sliit.lk} 



\begin{abstract}

\noindent We prove existence of weak and strong solutions and uniqueness for a viscous dyadic model driven by additive white noise in time using a path-wise approach. Existence of invariant measures also established and a simple balance relation among the mean rates of energy injection, dissipation and flux is derived and we investigate the asymptotic exponents $\zeta_{p}$ of the $p$-order structure functions.

\end{abstract}

\maketitle

\section [Introduction] {Introduction}

In the attempt of describing and understanding natural phenomena, mathematicians and physicists develop various models. Among these models, the Navier-Stokes equations are often used to describe the motion of fluids. Various forms of the Navier-Stokes equations have wide range of applications, such as, modeling weather, ocean currents, vortex, water flow in a pipe, blood flow in the body and air flow around a wing. Turbulence is where a fluid flow exhibits a chaotic behavior with time. It is widely believed that the Navier-Stokes equations are of great use when describing turbulence. The nonlinear term of the Navier-Stokes equations is the main contributor to the turbulence that the equations model and also the term that creates mathematical difficulties and challenges.\\

\noindent Most of the phenomena we encounter in daily life are not deterministic in nature: for instance, stock prices fluctuations are random. When modeling such phenomena, we incorporate a suitable stochastic process that accounts for the randomness, and we arrive at a stochastic ordinary differential equation or stochastic partial differential equation. The stochastic Navier-Stokes equations are of great interest as it is believed that they capture the dynamics of turbulent fluid flows in some regimes. Consequently, these equations are widely in the hope of obtaining stronger mathematical results that are not available under the deterministic case.\\

\noindent Despite the extensive use of the Navier-Stokes equations in practice, many mathematical questions remain open; the uniqueness of the global weak solution in three dimensions together with its regularity are among them. In the study of 3D turbulence, among the quantities of major interest are the asymptotic exponents $\zeta_{p}$ of the $p$-order structure functions. Numerical investigations as well as heuristics based on physical intuition have been extensively developed to support a multifractal structure of $\zeta_{p}$, see \cite{Biferale_2003} and the references therein. There is a general agreement on $\zeta_{p}<\frac{p}{3}$ for $p$ large and $\zeta_{3}=1$. The value of $\zeta_{2}$ is less clear as there is no analytical proof while certain numerical simulations indicate a value larger than $\frac{2}{3}$ opposed to $\zeta_{2}=\frac{2}{3}$ claimed by the  Kolmogorov theory. On the basis of previous works, see \cite{GOY_Flandoli}, we believe that some rigorous information on questions of turbulence theory could be better obtained from stochastic version of the equations of fluid dynamics. The advantage of the stochastic case is the major simplicity of balance laws between mean rates of energy injection, dissipation and flux. These are obtained by the It\^{o} formula and stochastic analysis. Our aim in this paper is to investigate such balance laws in one of the simplest settings related to fluid dynamics, shell models of turbulence.\\

\noindent Shell models are simplified Fourier systems with respect to the Navier-Stokes equations, where the interactions among different modes are preserved only between neighbors. In most of these shell models the interactions among the neighbors are limited to either three or five neighbors. Thus, shell models are drastic simplifications of the Navier-Stokes equations. In particular, the GOY and Sabra models are some of the most interesting and most popular examples of simplified phenomenological models of turbulence. Although departing from reality, they capture some essential statistical properties and features of turbulent flows, like the energy spectrum, the enstrophy cascade and the power law decay of the turbulent flows in some range of wave numbers-the inertial range, see \cite{Biferale_2003}.\\

\noindent We now describe how the content of this paper has been organized. We first introduce the shell models. We start by writing the Navier-Stokes equations in Fourier components; here, we assume that the Navier-Stokes equation is defined for periodic boundary conditions. In Fourier components the nonlinear term is represented by a term that contains infinitely many interactions among the neighboring terms and we arrive at shell models by considering only the local interactions between the neighbors. Thus, shell models are infinite system of coupled nonlinear scalar valued ordinary differential equations that preserve some properties of the original equation. Notice that geometry for the fluid is completely lost since we are dealing with scalar valued velocities instead of vector valued velocities. However, this does not affect the purpose of the current research. We then introduce the GOY model, Sabra model and dyadic model as examples for more general shell models. Here, we need to define the boundary condition for these models. We then introduce the functional setting associated to our models and define the associated operators and prove some of their properties.\\

In the next section, the stochastic version of the model is introduced. This is a viscous shell model driven by an additive white noise in time. We use the abstract form of the model and work in the setting defined earlier. We first prove the existence and uniqueness of strong (in the probabilistic sense) solutions. We present the proof in several steps. We use a pathwise argument and in our proof for existence, we use known results about the linear counterpart of the model, the Ornstein-Uhlenbeck process. Then, using a compactness theorem and following the approach used in \cite{GOY_Flandoli}, we show the existence and the uniqueness of solutions. Further, we give an estimate for the $p$-integrability of the solution via expectation, using a stopping time and applying the It\^{o} formula and the Burkholder-Davis-Gundy inequality. Then, we show that the solution is more regular, that is the solution belongs to $D(A)$ when the initial data is in $V$, a subspace of $H$ where $H$, $V$ and $D(A)$ are certain Hilbert spaces. Furthermore, improving the regularity of the solution, we prove that the solution belongs to $D(A)$ if the initial data is in $H$ following similar techniques used in \cite{Constantin_2006}.\\

\noindent The last section contains the results of our study of the asymptotic exponents $\zeta_{p}$ of $p$-order structure functions under a stochastic framework. Following the techniques introduced in \cite{GOY_Flandoli}, we first set the foundation for the study of these quantities by first establishing the existence of invariant measure for the shell model using the Feller property and applying the Krylov-Bogoliubov Theorem. Then, using the invariant measure, we establish a balance relation among the mean rates of energy injection, dissipation and flux between different scales. The power of stochastic tools became apparent when we use It\^{o} formula and stochastic calculus to get a balance relation in its simplest form. Then, we define the $p$-order structure functions and we work in some intermediate range of $n$s of the components of the solution, under suitable assumptions following the Kolmogorov theory. Based on $p$-order structure functions we define $\zeta_{p}$, the asymptotic exponents of order $p$. Our main interest is for $p=2$. All our results were derived assuming only the first component of the noise term is nonzero while all other noise terms are zero in our stochastic shell model. Based on the balance relation, we prove some results related to K41 theory including a monotonicity property and a boundedness property for mean rate of energy flux. We give sufficient and necessary conditions for $\zeta_{2}<\frac{2}{3}$ and $\zeta_{2}=\frac{2}{3}$. All these results were proved based on an assumption on the ratio related to structure functions and flux asymptotic exponent that is not derived from the balance relation.\\

\noindent All the results were proved for the dyadic model; however most of the proofs will work for the other shell models such as GOY model or Sabra model. An interesting open problem would be to prove that the assumption used can be derived from the balance relation but, unfortunately we are unable to answer that question for now. We can improve the regularity of solution in the Gevrey class (see, for instance \cite{Constantin_2006} for deterministic case), but we have not included such results in this paper.

\section[Shell Models and Functional Setting] {Shell Models and Functional Setting}

In this section, we first introduce the Navier-Stokes equations in Fourier components. Then, we introduce shell models as a phenomenological approximation of Navier-Stokes equations. A functional setting will be defined for these models, a crucial step in studying the well posedness of the stochastic version of these models.

\subsection{Navier-Stokes Equations in Fourier Space}

Navier-Stokes equations (NSEs) are of great importance in describing fluid flow. The NSEs governing the motion of a homogeneous, incompressible, viscous fluid in the three-dimensional torus are given by
\begin{equation}\label{NSE}
\frac{\partial \vec{u}(\vec{x},t)}{\partial t}= \nu\Delta \vec{u}(\vec{x},t) - (\vec{u}(\vec{x},t)\cdot\nabla)\vec{u}(\vec{x},t) -\nabla P(\vec{x},t)
\end{equation}
\noindent and
\begin{equation}\label{div_0}
\text{div } \vec{u}(\vec{x},t) = 0,
\end{equation}
\noindent assuming periodic boundary condition and initial condition, where $\vec{x}=(x_{1},x_{2},x_{3})\in [0,2\pi]^{3}$, $t\in[0,T]$, $\vec{u}(\vec{x},t)=(u_{1}(\vec{x},t),u_{2}(\vec{x},t),u_{3}(\vec{x},t))$ is the velocity field, $P=P(\vec{x},t)$ is the pressure field and $\nu >0$ is the viscosity coefficient. For further reading on NSEs, we refer to \cite{Temam_NSE_Functional_Analysis}, \cite{Temam_NSE_Numerical_Analysis} and \cite{Cannone}. Note that we may write \eqref{NSE} and \eqref{div_0} component-wise as
\begin{equation}\label{NSE_component_wise}\nonumber
\frac{\partial u_{j}}{\partial t} = \nu \sum_{n=1}^{3}\frac{\partial^{2} u_{j}}{\partial x_{n}^{2}}- \sum_{n=1}^{3}u_{n}\frac{\partial u_{j}}{\partial x_{n}} - \frac{\partial P}{\partial x_{j}}, \ \ \ j=1,2,3
\end{equation}
\noindent and
\begin{equation}\label{div_0_component_wise}\nonumber
\sum_{j=1}^{3}\frac{\partial u_{j}}{\partial x_{j}} = 0.
\end{equation}

\noindent Observe that taking divergence of \eqref{NSE} and using the condition \eqref{div_0}, we obtain
\begin{equation}\label{delta_P}\nonumber
\Delta P = -div (\vec{u}(\vec{x},t)\cdot\nabla)\vec{u}(\vec{x},t),
\end{equation}
\noindent or component-wise
\begin{equation}\label{delta_P_component_wise}\nonumber
\sum_{j=1}^{3}\frac{\partial^{2}P}{\partial x_{j}^{2}} = -\sum_{j=1}^{3}\frac{\partial}{\partial x_{j}}\sum_{n=1}^{3}u_{n}\frac{\partial u_{j}}{\partial x_{n}}.
\end{equation}
\noindent In order to write NSEs \eqref{NSE} and \eqref{div_0} in Fourier components, consider the Fourier basis $(e^{i\vec{k}\cdot \vec{x}})_{\vec{k}\in \mathbb{Z}^{3}}$ in $L^{2}([0,2\pi]^{3})$. We have
\begin{equation}\label{u_Fourier}
\vec{u}(\vec{x},t) = \sum_{\vec{k}\in \mathbb{Z}^{3}}\vec{u}_{\vec{k}}(t)e^{i\vec{k}\cdot \vec{x}}, \ \ \ \ \vec{u}_{\vec{k}}(t)\in \mathbb{C}^{3}
\end{equation}
\noindent and
\begin{equation}\label{P_Fourier}
P(\vec{x},t) = \sum_{\vec{k}\in \mathbb{Z}^{3}}P_{\vec{k}}(t)e^{i\vec{k}\cdot \vec{x}}, \ \ \ \ P_{\vec{k}}(t)\in \mathbb{C}.
\end{equation}
\noindent Using \eqref{u_Fourier} and \eqref{P_Fourier}, we write each term in \eqref{NSE} and \eqref{div_0} in Fourier components as
\begin{equation}\label{partial_u_partial_t_Fourier}
\frac{\partial \vec{u}(\vec{x},t)}{\partial t} = \sum_{\vec{k}\in \mathbb{Z}^{3}}\frac{\partial \vec{u}_{\vec{k}}(t)}{\partial t}e^{i\vec{k}\cdot \vec{x}},
\end{equation}
\begin{equation}\label{Delta_u}
\Delta \vec{u}(\vec{x},t)= - \sum_{\vec{k}\in \mathbb{Z}^{3}}|\vec{k}|^{2}\vec{u}_{\vec{k}}(t)e^{i\vec{k}\cdot \vec{x}},
\end{equation}
\begin{equation}\label{nabla_P}
(\nabla P)_{j} = i \sum_{\vec{k}\in \mathbb{Z}^{3}} P_{\vec{k}}(t)\vec{k}_{j}e^{i\vec{k}\cdot \vec{x}},\ \ \ \ j=1,2,3,
\end{equation}
\noindent and
\begin{align}\label{non_linear_term_Fourier}\nonumber
(\vec{u}(\vec{x},t)\cdot\nabla)\vec{u}(\vec{x},t) &= i\sum_{n=1}^{3}\sum_{\vec{k}\in \mathbb{Z}^{3}}\vec{u}_{\vec{k}_{n}}(t)e^{i\vec{k}\cdot \vec{x}}\sum_{\vec{k}\in \mathbb{Z}^{3}}\vec{u}_{\vec{k}_{j}}(t)\vec{k}_{n}e^{i\vec{k}\cdot \vec{x}}\nonumber
\\ & = i\sum_{\vec{k}\in \mathbb{Z}^{3}}\sum_{\vec{l}+\vec{h}=\vec{k}}(\vec{l}\cdot \vec{u}_{\vec{h}}(t))\vec{u}_{\vec{l}}(t)\ e^{i\vec{k}\cdot \vec{x}}.\nonumber
\end{align}
\noindent Since $div\ \vec{u}(\vec{x},t)=0$ we have $\vec{k}\cdot \vec{u}_{\vec{k}}(t)=0$ for all $\vec{k}\in \mathbb{Z}^{3}$. Thus, we have
\begin{equation}\label{non_linear_term_Fourier_2}
(\vec{u}(\vec{x},t)\cdot\nabla)\vec{u}(\vec{x},t) = i\sum_{\vec{k}\in \mathbb{Z}^{3}}\sum_{\vec{l}+\vec{h}=\vec{k}}(\vec{k}\cdot \vec{u}_{\vec{h}}(t))\vec{u}_{\vec{l}}(t)\ e^{i\vec{k}\cdot \vec{x}}.
\end{equation}
\noindent Note that
\begin{equation}\nonumber
\Delta P = - \sum_{\vec{k}\in \mathbb{Z}^{3}} |\vec{k}|^{2}P_{\vec{k}}(t)e^{i\vec{k}\cdot \vec{x}},
\end{equation}
\noindent and using \eqref{delta_P} and \eqref{non_linear_term_Fourier_2} we get
\begin{equation}\label{P_substitution}
P_{\vec{k}}(t) = -\frac{\sum_{\vec{l}+\vec{h}=\vec{k}}(\vec{k}\cdot \vec{u}_{\vec{h}}(t))(\vec{k}\cdot \vec{u}_{\vec{l}}(t))}{|\vec{k}|^{2}}.
\end{equation}
\noindent Now using \eqref{partial_u_partial_t_Fourier}, \eqref{Delta_u}, \eqref{nabla_P}, \eqref{non_linear_term_Fourier_2}, \eqref{P_substitution} together with \eqref{NSE} we obtain the following infinite dimensional coupled system of differential equations
\begin{equation}\label{NSE_Fourier}
\frac{\partial \vec{u}_{\vec{k}}(t)}{\partial t} = -\nu |\vec{k}|^{2}\vec{u}_{\vec{k}}(t)-i\sum_{\vec{l}+\vec{h}=\vec{k}}(\vec{k}\cdot \vec{u}_{\vec{h}}(t))\bigg[\vec{u}_{\vec{l}}(t)-\frac{(\vec{k}\cdot \vec{u}_{\vec{l}}(t))}{|\vec{k}|^{2}}\vec{k}\bigg]
\end{equation}
\noindent such that
\begin{equation}\nonumber
\vec{k}\cdot \vec{u}_{\vec{k}}(t)=0 \text{ for all } \vec{k}\in \mathbb{Z}^{3}.
\end{equation}

\subsection{Shell Models of Turbulence}


Shell Models are simplified phenomenological models that capture some properties of turbulent fluid flows. They are simplified versions of \eqref{NSE_Fourier}. Shell models are described by an infinite dimensional system of coupled ordinary differential equations that contain nonlinearities similar to \eqref{NSE_Fourier}. The main difference with \eqref{NSE_Fourier} is that shell models contain only local interactions among the shell variables. In particular, most of the well-known shell models contain interactions among three or five nearest neighbors. Thus, shell models are drastically simplified models compared to NSEs and this very reduced number of interactions and parametrization of the fluctuation of a turbulent field in each shell by small number of variables of shell models give us computational advantages. \\

\noindent It is natural to consider NSEs \eqref{NSE} with the presence of external forces, see for instance, \cite{Temam_NSE_Functional_Analysis}. NSEs with forcing can be written in Fourier components and for further reading see for instance \cite{Marco_Romito}. It is believed (see for instance \cite{SNSE_Flandoli}) that questions on turbulence theory could be answered more rigorously by taking into account the randomness as well. NSEs have been studied under different kinds of stochastic forcing, see for instance \cite{SNSE_Rozovskii}. In a similar line of thoughts, shell models have been also studied when driven by an external force which is deterministic or stochastic. Among the well known shell models are the Gledzer-Okhitani-Yamada (GOY) model, introduced in \cite{GOY_89}, and the ``Sabra'' model, introduced in \cite{Sabra_98}. For further studies on GOY model we refer \cite{GOY_Flandoli} and for Sabra model we refer \cite{Constantin_2006} and \cite{Constantin_2007}. For further reading on some stochastic shell models, we refer to very recent papers \cite{Annie_Millet} and \cite{Attractors_Flandoli}. The Dyadic model \eqref{shell_model_differential_form_componentwise} has been studied previously in \cite{Energy_Flandoli}, \cite{Cheskidov_attractor}, \cite{Cheskidov_conjecture}, \cite{Cheskidov_blowup} \cite{Cheskidov_vanishing} and \cite{Uniqueness_Flandoli}.

\subsection{GOY, Sabra and Dyadic Models}

\noindent We will be dealing with the following infinite system of nonlinear ODEs,
\begin{equation}\label{shell_model_general}
du_{n} + \nu k_{n}^{2}u_{n}dt = (B(u,u))_{n}dt + \sigma_{n}d\beta_{n} \text{ for }n=1,2,\ldots\ .
\end{equation}
\noindent For a given $k_{0}>0$, $k_{n}=k_{0}\lambda^{n}$ denotes the $n^{text{th}}$ one dimensional ``wave number'' where $n$ represents the ``shell index'' and $\lambda>1$ being the ``shell spacing parameter''. The sequence of functions $(u_{n}(t))_{n\geq0}$ on $[0,T]$, describes the average value of the Navier-Stokes Fourier components of velocities in a ``shell''. The interval $(k_{n},k_{n+1})$ is known as a shell. The constant $\nu>0$ is the viscosity.  Additive noise is denoted by $(\beta_{n})_{n\geq1}$, a sequence of real valued independent Brownian motions defined on a probability space $(\Omega, \mathcal{F},\mathcal{P})$ with $(\sigma_{n})_{n\geq1}$.\\

\noindent The nonlinear term $(B(u,u))_{n}$ has different forms depending on the model. For the GOY model we have that
\begin{equation}\nonumber
(B(u,u))_{n} = ik_{n}\bigg(\frac{1}{4}\overline{u}_{n-1}\overline{u}_{n+1}-\overline{u}_{n+1}\overline{u}_{n+2} +\frac{1}{8}\overline{u}_{n-1}\overline{u}_{n-2}\bigg) \text{ for }n=1,2,\ldots
\end{equation}
\noindent where $u_{n}(t)$ are complex valued functions over $[0,T]$ and $\overline{u}_{n}$ denotes the complex conjugate of $u_{n}$. In this case \eqref{shell_model_general} has the boundary conditions $u_{-1}(t)=u_{0}(t)=0$.\\

\noindent For the Sabra shell model, the nonlinear term has the form
\begin{equation}\nonumber
(B(u,u))_{n} = i(ak_{n+1}u_{n+2}\overline{u}_{n+1}+bk_{n}u_{n+1}\overline{u}_{n-1} -ck_{n-1}u_{n-1}u_{n-2}) \text{ for }n=1,2,\ldots
\end{equation}
\noindent together with the boundary conditions $u_{-1}(t)=u_{0}(t)=0$. The parameters $a,b$ and $c$ are real, $u_{n}(t)$ are complex valued functions over $[0,T]$ and $\overline{u}_{n}$ denotes the complex conjugate of $u_{n}$. Defining the energy as
$$E=\sum_{n=1}^{\infty}|u_{n}|^{2},$$
and requiring conservation of energy for the inviscid and unforced case, we obtain the following relation among the parameters
$$a+b+c=0,$$
known as the energy conservation condition.\\

\noindent The following nonlinear term gives rise to the dyadic model
\begin{equation}\nonumber
(B(u,u))_{n} = (k_{n-1}u_{n-1}^{2}-k_{n}u_{n}u_{n+1}) \text{ for }n=1,2,\ldots
\end{equation}
\noindent with the boundary condition $u_{0} = 0$ and $u_{n}(t)$ are real valued functions over $[0,T]$. The ``shell spacing parameter'' $\lambda$ is selected to be 2, hence the name ``dyadic model''.

\subsection{Functional Setting}

The functional setting will be given for the dyadic model
\begin{equation}\label{shell_model_differential_form_componentwise}
du_{n} + \nu k_{n}^{2}u_{n}dt = (k_{n-1}u_{n-1}^{2}-k_{n}u_{n}u_{n+1})dt + \sigma_{n}d\beta_{n} \text{ for }n=1,2,\ldots\
\end{equation}
\noindent with the boundary condition $u_{0} = 0$ where $u_{n}(t)$ are real valued random variables. However, most of the properties will be true for all other shell models defined previously.\\

\noindent We will introduce the following spaces over $\mathbb{R}$ and then define the operators on them so that we can write \eqref{shell_model_differential_form_componentwise} in an abstract form.\\
\noindent Let $u=(u_{1},u_{2},u_{3},\ldots)\in\mathbb{R}^{\infty}$. Let $H$ be the Hilbert space defined by
\begin{equation}\nonumber
H = \bigg\{u:\sum_{n=1}^{\infty}u_{n}^{2}<\infty\bigg\}.
\end{equation}
\noindent We have the inner product defined on $H$ by
\begin{equation}\nonumber
\langle u,v \rangle_{H}=\sum_{n=1}^{\infty}u_{n}v_{n} \text{ for all } u,v \in H
\end{equation}
\noindent and the norm is given by
\begin{equation}\nonumber
|u|_{H}^{2}=\sum_{n=1}^{\infty}u_{n}^{2} \text{ for all } u\in H.
\end{equation}
\noindent We define the operator $A:D(A)\subset H \rightarrow H$ as
\begin{equation}\label{A}
Au = (k_{1}^{2}u_{1}, k_{2}^{2}u_{2}, k_{3}^{2}u_{3}, \ldots ) \text{ for all } u \in D(A).
\end{equation}
\noindent The linear operator $A$ is self-adjoint and strictly positive definite operator. The domain of $A$, denoted by $D(A)$, is the subspace defined as
\begin{equation}\nonumber
D(A) = \bigg\{u:\sum_{n=1}^{\infty}k_{n}^{4}u_{n}^{2}<\infty\bigg\}.
\end{equation}
\noindent We have $D(A)$ dense in $H$, and $D(A)$ is a Hilbert space with the graph norm $|Au|_{H}$ defined by
\begin{equation}\nonumber
|u|_{D(A)}^{2}=\sum_{n=1}^{\infty}k_{n}^{4}u_{n}^{2} \text{ for all } u\in D(A).
\end{equation}
\noindent Since $A$ is positive definite operator we can define the powers of $A$, denoted by $A^{s}$ for all $s\in \mathbb{R}$. We define
\begin{equation}\nonumber
A^{s}u = (k_{1}^{2s}u_{1}, k_{2}^{2s}u_{2}, \ldots) \text{ for all } u\in D(A^{s})
\end{equation}
\noindent where $D(A^{s})$ is defined as
\begin{equation}\nonumber
V_{s}:=D(A^{s/2})= \bigg\{u :\sum_{n=1}^{\infty}k_{n}^{2s}u_{n}^{2}<\infty\bigg\}.
\end{equation}
\noindent The spaces $V_{s}$ are Hilbert spaces with the inner product defined as
\begin{equation}\nonumber
\langle u,v \rangle_{V_{s}}=\langle A^{s/2}u, A^{s/2}v \rangle_{H} \text{ for all } u,v \in V_{s},
\end{equation}
\noindent and we have the norm on $V_{s}$ defined as
\begin{equation}\nonumber
|u|_{V_{s}}^{2}=\sum_{n=1}^{\infty}k_{n}^{2s}u_{n}^{2} \text{ for all } u\in V_{s}.
\end{equation}
\noindent For all $s\in\mathbb{R}$, we have $V_{s}^{'}$, the dual space of $V_{s}$, is $V_{-s}$. We define the dual pairing between $V_{s}$ and $V_{-s}$ as
\begin{equation}\nonumber
\langle u,v \rangle_{V_{-s}\times V_{s}} = \langle A^{-s/2}u, A^{s/2}u \rangle_{H} = \sum_{n=1}^{\infty}u_{n}v_{n} \text{ for all } u\in V_{-s} \text{ and } u\in V_{s}.
\end{equation}
\noindent Note that we are interested in certain special cases of $V_{s}$. When $s=0$ we recover $H$ and when $s=2$ we recover $D(A)$. For $s=1$ we write $V$ and for $s=-1$ we write $V'$. We have the following inner products and norms defined on the Hilbert spaces $V$ and $V'$ respectively. On the Hilbert space $V$ we have
\begin{equation}\nonumber
|u|_{V}^{2}=\sum_{n=1}^{\infty}k_{n}^{2}u_{n}^{2},
\end{equation}
\noindent and
\begin{equation}\nonumber
\langle u,v \rangle_{V}=\sum_{n=1}^{\infty}u_{n}v_{n} \text{ for all } u,v \in V
\end{equation}
\noindent and on the Hilbert space $V'$ we have
\begin{equation}\nonumber
|u|_{V'}=\sum_{n=1}^{\infty}k_{n}^{-2}u_{n}^{2},
\end{equation}
\noindent and
\begin{equation}\nonumber
\langle u,v \rangle_{V'}=\sum_{n=1}^{\infty}u_{n}v_{n} \text{ for all } u,v \in V'.
\end{equation}
\noindent Also, note that we can extend the linear operator $A$ to $A(\cdot):V\rightarrow V'$ since $D(A)$ is dense in $V$.
Further, we observe that the following embedding relation (with compactness) holds;
\begin{equation}\nonumber
V_{s}\subset V_{0}=H\subset V_{-s} \text{ for all } s>0.
\end{equation}
\noindent Finally, for all $u\in H$ and $v\in V$, we define the bilinear operator $B(\cdot,\cdot):H\times V \rightarrow H$ as
\begin{equation}\label{B}
(B(u,v))_{n} = k_{n-1}u_{n-1}v_{n-1}-k_{n}u_{n}v_{n+1}.
\end{equation}
\noindent Remark that we can define $B$ as $B(\cdot,\cdot):V\times H \rightarrow H$ as well. Further, as we show in the next section, we can extend $B$ to $B(\cdot,\cdot):H\times H \rightarrow V'$.

\subsection{Properties of the Bilinear Operator}

\noindent We have the following properties for the bilinear operator $B$.
\begin{proposition}
Let $u\in V$ and $v\in H$. Then
\begin{equation}\label{B_zero_property}
\langle B(u,v),v \rangle_{H} = 0.
\end{equation}
\end{proposition}
\begin{proof}
\begin{align}
\langle B(u,v),v \rangle_{H} &= \sum_{n=1}^{\infty}(B(u,v))_{n}v_{n} \nonumber
\\& = \sum_{n=1}^{\infty}(k_{n-1}u_{n-1}v_{n-1}-k_{n}u_{n}v_{n+1})v_{n} \nonumber
\\& = \sum_{n=2}^{\infty}k_{n-1}u_{n-1}v_{n-1}v_{n}-\sum_{n=1}^{\infty}k_{n}u_{n}v_{n}v_{n+1} \ \text{ (since $u_{0}=0$}) \nonumber
\\& = \sum_{n=1}^{\infty}k_{n}u_{n}v_{n}v_{n+1}-\sum_{n=1}^{\infty}k_{n}u_{n}v_{n}v_{n+1} \nonumber
\\& = 0. \nonumber
\end{align}
\end{proof}
\noindent We note that \eqref{B_zero_property} holds when $u\in H$ and $v\in V$ as well. We now give the following symmetry property of the operator $B$.
\begin{proposition}
Let $u,v\in H$ and $w\in V$. Then
\begin{equation}\nonumber
\langle B(u,v),w \rangle_{V'\times V} = - \langle B(u,w),v \rangle_{H}.
\end{equation}
\end{proposition}
\begin{proof}
\begin{align}
\langle B(u,v),w \rangle_{V'\times V} &= \sum_{n=1}^{\infty}(B(u,v))_{n}w_{n} \nonumber
\\& = \sum_{n=1}^{\infty}(k_{n-1}u_{n-1}v_{n-1}-k_{n}u_{n}v_{n+1})w_{n} \nonumber
\\& = \sum_{n=1}^{\infty}k_{n-1}u_{n-1}v_{n-1}w_{n}-\sum_{n=1}^{\infty}k_{n}u_{n}v_{n+1}w_{n}  \nonumber
\end{align}
\begin{align}
& = \sum_{n=2}^{\infty}k_{n-1}u_{n-1}v_{n-1}w_{n}-\sum_{n=1}^{\infty}k_{n-1}u_{n-1}v_{n}w_{n-1} \text{ since $u_{0}=0$}\nonumber
\\& = \sum_{n=1}^{\infty}k_{n}u_{n}v_{n}w_{n+1}-\sum_{n=1}^{\infty}k_{n-1}u_{n-1}v_{n}w_{n-1} \nonumber
\\& = -\sum_{n=1}^{\infty}(k_{n-1}u_{n-1}w_{n-1}-k_{n}u_{n}w_{n+1})v_{n} \nonumber
\\& = -\sum_{n=1}^{\infty}(B(u,w))_{n}v_{n} \nonumber
\\& = -\langle B(u,w),v\rangle_{H}.\nonumber
\end{align}
\end{proof}
\noindent Now, we state a general Lemma for the operator $B$.
\begin{lemma}
Let $d,s\in \mathbb{R}$. Then $B$ can be defined as $B(\cdot,\cdot):V_{2d-2s}\times V_{1+2s}\rightarrow V_{2d}$. Moreover, there exists a constant $C>0$ depending on $d$ and $s$ such that
\begin{equation}\nonumber
|B(u,v)|_{V_{2d}} \leq C_{d,s} |u|_{V_{2d-2s}}|v|_{V_{1+2s}}
\end{equation}
for all $u\in V_{2d-2s}$ and $v\in V_{1+2s}$.
\end{lemma}
\begin{proof}
\noindent We calculate the norm in $V_{2d}$ as
\begin{align}\label{B_V_2d_norm}
&|B(u,v)|_{V_{2d}}^{2} \nonumber
\\&= \sum_{n=1}^{\infty}k_{n}^{4d}|(B(u,v))_{n}|^{2}\nonumber
\\& = \sum_{n=1}^{\infty}k_{n}^{4d}|k_{n-1}u_{n-1}v_{n-1} - k_{n}u_{n}v_{n+1}|^{2}\nonumber
\\&=\sum_{n=1}^{\infty}k_{n}^{4d}|k_{n-1}^{2}u_{n-1}^{2}v_{n-1}^{2}-2k_{n-1}k_{n}u_{n-1}u_{n}v_{n-1}v_{n+1}+k_{n}^{2}u_{n}^{2}v_{n+1}^{2}|\nonumber
\\&\leq \sum_{n=1}^{\infty}k_{n-1}^{2}k_{n}^{4d}u_{n-1}^{2}v_{n-1}^{2} + \sum_{n=1}^{\infty}2k_{n-1}k_{n}^{4d+1}|u_{n-1}u_{n}v_{n-1}v_{n+1}|+ \sum_{n=1}^{\infty}k_{n}^{4d+2}u_{n}^{2}v_{n+1}^{2}.
\end{align}\\

\noindent Now we consider the three terms in \eqref{B_V_2d_norm}. The first term is

\begin{align}\label{first_part_in_B_bounded_estimate}
\sum_{n=1}^{\infty}k_{n-1}^{2}k_{n}^{4d}u_{n-1}^{2}v_{n-1}^{2} &=2^{4d}\sum_{n=1}^{\infty}k_{n-1}^{4d}k_{n-1}^{2}u_{n-1}^{2}v_{n-1}^{2}\text{ since $k_{n}=2k_{n-1}$ and $k_{n}^{4d}=2^{4d}k_{n-1}^{4d}$}\nonumber
\\& =2^{4d}\sum_{n=1}^{\infty}k_{n-1}^{4d-4s}k_{n-1}^{2+4s}u_{n-1}^{2}v_{n-1}^{2} \nonumber
\\&\leq 2^{4d}\sum_{n=1}^{\infty}\bigg[\sup_{n}k_{n-1}^{4d-4s}u_{n-1}^{2}\bigg]k_{n-1}^{2+4s}v_{n-1}^{2}\nonumber
\\& = 2^{4d}\bigg[\sup_{n}k_{n-1}^{4d-4s}u_{n-1}^{2}\bigg]\sum_{n=1}^{\infty}k_{n-1}^{2+4s}v_{n-1}^{2}\nonumber
\\&\leq 2^{4d}\sum_{n=1}^{\infty}k_{n-1}^{4d-4s}u_{n-1}^{2}\sum_{n=1}^{\infty}k_{n-1}^{2+4s}v_{n-1}^{2}\nonumber
\\& = 2^{4d}|u|_{V_{2d-2s}}^{2}|v|_{V_{1+2s}}^{2}.
\end{align}
\noindent The second term in \eqref{B_V_2d_norm} is
\begin{align}\label{second_part_in_B_bounded_estimate}
&\sum_{n=1}^{\infty}2k_{n-1}k_{n}^{4d+1}|u_{n-1}u_{n}v_{n-1}v_{n+1}| \nonumber \\&=\sum_{n=1}^{\infty}k_{n}^{4d+1}k_{n}|u_{n-1}u_{n}v_{n-1}v_{n+1}| \nonumber
\\& \leq\sum_{n=1}^{\infty}k_{n}^{4d+2}[u_{n-1}^{2}v_{n-1}^{2}+u_{n}^{2}v_{n+1}^{2}] \nonumber
\\&=\sum_{n=1}^{\infty}k_{n}^{4d+2}u_{n-1}^{2}v_{n-1}^{2} +\sum_{n=1}^{\infty}k_{n}^{4d+2}u_{n}^{2}v_{n+1}^{2}\nonumber
\\&= 2^{4d+2}\sum_{n=1}^{\infty}k_{n-1}^{4d+2}u_{n-1}^{2}v_{n-1}^{2} +\sum_{n=1}^{\infty}k_{n}^{4d-4s}u_{n}^{2}k_{n}^{2+4s}v_{n+1}^{2}\nonumber
\\& (\text{since }k_{n}=2k_{n-1}, k_{n}^{4d+2}=2^{4d+2}k_{n-1}^{4d+2})\nonumber
\\&= 2^{4d+2}\sum_{n=1}^{\infty}k_{n-1}^{4d-4s}u_{n-1}^{2}k_{n-1}^{2+4s}v_{n-1}^{2} +\bigg(\frac{1}{2}\bigg)^{2+4s}\sum_{n=1}^{\infty}k_{n}^{4d-4s}u_{n}^{2}k_{n+1}^{2+4s}v_{n+1}^{2}\nonumber
\\& \bigg[\text{since }k_{n}=\bigg(\frac{1}{2}\bigg)k_{n+1}, k_{n}^{2+4s}=\bigg(\frac{1}{2}\bigg)^{2+4s}k_{n+1}^{2+4s}\bigg]\nonumber
\\& \leq 2^{4d+2}\sup_{n}\bigg[k_{n-1}^{4d-4s}u_{n-1}^{2}\bigg]\sum_{n=1}^{\infty}k_{n-1}^{2+4s}v_{n-1}^{2} +\bigg(\frac{1}{2}\bigg)^{2+4s}\sup_{n}\bigg[k_{n}^{4d-4s}u_{n}^{2}\bigg]\sum_{n=1}^{\infty}k_{n+1}^{2+4s}v_{n+1}^{2}\nonumber
\\& \leq \max\bigg\{2^{4d+2},\bigg(\frac{1}{2}\bigg)^{2+4s}\bigg\}[|u|_{V_{2d-2s}}^{2}|v|_{V_{1+2s}}^{2}+|u|_{V_{2d-2s}}^{2}|v|_{V_{1+2s}}^{2}].
\end{align}
\noindent The third term in \eqref{B_V_2d_norm},
\begin{align}\label{third_part_in_B_bounded_estimate}
\sum_{n=1}^{\infty}k_{n}^{4d+2}u_{n}^{2}v_{n+1}^{2}&=\sum_{n=1}^{\infty}k_{n}^{4d-4s}u_{n}^{2}k_{n}^{2+4s}v_{n+1}^{2}\nonumber
\\& \leq \bigg(\frac{1}{2}\bigg)^{2+4s} \sup_{n}\bigg[k_{n}^{4d-4s}u_{n}^{2}\bigg]\sum_{n=1}^{\infty}k_{n+1}^{2+4s}v_{n+1}^{2}\nonumber
\\& \leq \bigg(\frac{1}{2}\bigg)^{2+4s} |u|_{V_{2d-2s}}^{2}|v|_{V_{1+2s}}^{2}.
\end{align}
\noindent The result follows from \eqref{B_V_2d_norm}, \eqref{first_part_in_B_bounded_estimate}, \eqref{second_part_in_B_bounded_estimate} and \eqref{third_part_in_B_bounded_estimate} with $C_{d,s}=\max\bigg\{2^{4d+2},\bigg(\frac{1}{2}\bigg)^{2+4s}\bigg\}$.
\end{proof}
\noindent We immediately have the following corollary.
\begin{coro}\label{B_boundedness_property}
There exists a constant $C>0$ such that
\begin{enumerate}
\item $|B(u,v)|_{H} \leq C |u|_{V}|v|_{H} \text{ for all } u\in V, v\in H$,
\item $|B(u,v)|_{H} \leq C |u|_{H}|v|_{V} \text{ for all } u\in H, v\in V$,
\item $|B(u,v)|_{V'} \leq C |u|_{H}|v|_{H} \text{ for all } u,v\in H$.
\end{enumerate}
\end{coro}

\section[Well-posedness of Stochastic Shell Models of Turbulence] {Well-posedness of Stochastic Shell Models of Turbulence}

\noindent We start this section by recalling some results for a linear stochastic differential equation driven by an additive noise.

\subsection{Stochastic Linear Differential Equations with Additive Noise}

\noindent Let $(\Omega, \mathcal{F}, P)$ be a given probability space, $\mathcal{F}_{t}$ be the associated filtration for $t\geq 0$, and let $H$ and $U$ be two Hilbert spaces. Consider a $Q$-Wiener process $W$ defined on $(\Omega, \mathcal{F}, P)$ with covariance operator $Q\in \mathcal{L}(U,U)$. Consider the SDE
\begin{equation}\label{linear_SDE}
dX(t)=AX(t)dt+f(t)dt + GdW(t),
\end{equation}
\begin{equation}
X(0)=\xi,
\end{equation}
\noindent where $A:D(A)\subset H\rightarrow H$ and $G:U\rightarrow H$ are linear operators and $f$ is a $H$-valued stochastic process.
\begin{definition}\emph{(Weak Solution)}\nonumber
We say that an $H$-valued predictable process $X(t)$, where $t\in [0,T]$ is a \emph{weak solution} of the equation \eqref{linear_SDE} if the sample paths of $X(\cdot)$ are $P$-a.s. integrable and if for all $\phi \in D(A^{*})$ and for all $t\in [0,T],$ we have
\begin{equation}
\langle X(t),\phi \rangle = \langle \xi,\phi \rangle + \int_{0}^{t} \langle X(s),A^{*}\phi \rangle ds + \langle f(s),\phi \rangle ds + \langle GW(t),\phi \rangle \ \ P\text{-a.s.}
\end{equation}
\end{definition}
\noindent Consider the following assumptions.
\begin{enumerate}
\item The operator $A$ generates a strongly continuous semigroup $S(\cdot)$ in $H$ [i.e., the deterministic Cauchy problem $u'(t)=Au(t)$ with $u(0)=\xi\in H$ is uniformly well-posed],
\item $G\in \mathcal{L}(U:H)$ [i.e., $G$ is bounded],
\item $f$ is a predictable process with integrable sample paths on $[0,T]$ for any $T>0$,
\item $\xi$ is $\mathcal{F}_{0}$-measurable,
\item $\displaystyle \int_{0}^{T}|S(r)G|^{2}_{L_{2}^{0}}dr =  \displaystyle \int_{0}^{T}Tr[S(r)GQG^{*}S^{*}(r)]dr <\infty.$ (See remark \ref{space_of_all_Hilbert-Schmidt_operators}.)
\end{enumerate}

\begin{remark}\label{space_of_all_Hilbert-Schmidt_operators}(\emph{The Space $L_{2}^{0}$})
Let $(\Omega, \mathcal{F}, P)$ be a probability space, $H$, $U$ be a separable Hilbert space, $\{e_{j}\}$ be a complete orthonormal system in $H$ and we are given a $H$-valued $Q$-Wiener process on $[0,T]$. Suppose $\lambda_{j}>0$ for all $j=1,2,\cdots$. Fix $T<\infty$ and let $L=L(H,U).$ We define the space $$H_{0}=Q^{1/2}(H),$$ endowed with the inner product $$\langle u, v \rangle_{0}=\sum_{j=0}^{\infty}\frac{1}{\lambda_{j}}\langle u, e_{j} \rangle\langle v, e_{j} \rangle=\langle Q^{-1/2}u, Q^{-1/2}v \rangle.$$
$H_{0}$ is a Hilbert space and a subspace of $H$. We now introduce space of all Hilbert-Schmidt operators from $H_{0}$ into $H$, denoted by $L_{2}^{0}=L_{2}(H_{0},H)$. It is a separable Hilbert space equipped with the norm 
\begin{align}
\|\psi\|^{2}_{L_{2}^{0}}&=\sum_{m,n=1}^{\infty}|\langle\psi g_{m},f_{n}\rangle|^{2}\nonumber
\\ &=\sum_{m,n=1}^{\infty}\lambda_{m}|\langle\psi e_{m},f_{n}\rangle|^{2}\nonumber
\\ &=\|\psi Q^{1/2}\|^{2}\nonumber
\\ &=\text{Tr}(\psi Q \psi^{*})\nonumber
\end{align}
where $\{e_{j}\}$, $\{f_{j}\}$ and $\{g_{j}\}$ (with $g_{j}=\sqrt{\lambda_{j}}e_{j}$ where $j=1,2,\cdots$) are complete orthonormal bases in $H$, $U$, and $H_{0}$ respectively. We have $L\subset L_{2}^{0}$ and in particular, the space $L_{2}^{0}$ contains genuinely unbounded operators on $U$.
\end{remark}
\begin{theorem}\label{linear_SODE_existence}
Under the above assumptions 1,2,3,4, and 5, the SDE \eqref{linear_SDE} has exactly one weak solution given by the formula
\begin{equation}\nonumber
X(t) = S(t)\xi + \int_{0}^{t}S(t-s)f(s)ds + \int_{0}^{t}S(t-s)GdW(s) \ \ \text{where } t\in [0,T].
\end{equation}
\noindent Moreover, P-a.s. $X\in C([0,T],H).$
\end{theorem}
\begin{proof}
\noindent For a proof of Theorem \eqref{linear_SODE_existence} we refer \cite{DaPrato}, page 121.
\end{proof}

\subsection{Wellposedness of Shell Models}

\subsubsection{Existence of Solution}

Using \eqref{A} and \eqref{B}, we write \eqref{shell_model_differential_form_componentwise} in the abstract form. We have
\begin{equation}\label{shell_model_differential_form_infinite_dimensional}
du(t) +\nu Au(t) = B(u(t),u(t))dt + dW(t), \ t\geq 0
\end{equation}
\noindent where $\{W(t)\}_{t\geq 0}$ is an $H$-valued Brownian motion defined on a filtered probability space $(\Omega, \mathcal{F}, (\mathcal{F})_{t\geq 0},P)$ with nuclear covariance operator $Q$. We assume an initial condition given by an $\mathcal{F}_{0}$ measurable random variable $u(0): \Omega \rightarrow H$.
\begin{definition}
We say that \eqref{shell_model_differential_form_infinite_dimensional} is written in its integral form if
\begin{align}\label{shell_model_integral_weak_form_infinite_dimensional}
\langle u(t),\phi(t)\rangle_{H} + \nu \int_{0}^{t}\langle u(s),A(\phi(s))\rangle_{H}ds &= \langle u(0),\phi(t)\rangle_{H} + \int_{0}^{t}\langle B(u(s),\phi(s)),u(s)\rangle_{H}ds\nonumber
\\& + \langle W(t),\phi(t)\rangle_{H}
\end{align}
for all $\phi\in D(A)$ and for all $t\in[0,T].$
\end{definition}
\begin{definition}
We say that an adapted process $u\in C([0,T]:H)$ is a weak solution of \eqref{shell_model_differential_form_infinite_dimensional} if $u$ satisfies \eqref{shell_model_integral_weak_form_infinite_dimensional} for all $\phi\in D(A)$ and for some $T>0$. We say this solution is global if such $u\in C([0,T]:H)$ $P$-a.s. for any $T>0$.
\end{definition}
\begin{theorem}\label{existence_weak_solution}
Let $u(0):\Omega \rightarrow H$, a $\mathcal{F}_{0}$-measurable random variable be an initial condition. There exists a global weak solution of \eqref{shell_model_differential_form_infinite_dimensional}.
\end{theorem}
\begin{proof}The proof is presented in several steps.\\

\noindent \underline{\textbf{Step 1}}\\

\noindent We use a pathwise approach. Fix $u(0)\in H$ and $\omega \in C^{\alpha}([0,T]:H)$ for some $0\leq \alpha <\frac{1}{2}$. Then, we have the deterministic equation
\begin{align}\label{shell_model_integral_weak_form_pathwise_infinite_dimensional}
\langle u(t),\phi(t)\rangle_{H} + \nu \int_{0}^{t}\langle u(s),A(\phi(s))\rangle_{H}ds &= \langle u(0),\phi(t)\rangle_{H}+\int_{0}^{t}\langle B(u(s),\phi(s)),u(s)\rangle_{H}ds\nonumber
\\& + \langle \omega(t),\phi(t)\rangle_{H}
\end{align}
\noindent where $\phi \in D(A)$.\\

\noindent \underline{\textbf{Step 2}}\\

\noindent Let $m\in \mathbb{N}$ and let
\begin{equation}\nonumber
H_{m}=\{u\in H\ : \ u_{j}=0 \text{ for all } j>m\}
\end{equation}
\noindent and consider the finite dimensional orthogonal projections $\pi_{m}:H\rightarrow H_{m}$. Thus, for $u\in H$ we have $\pi_{m}(u) = (u_{1},u_{2},\ldots ,u_{m},0,0\ldots)$ and this is denoted by $u^{(m)}$. In particular we have that $$|u^{(m)}|_{H}=|\pi_{m}u|_{H}\leq|u|_{H}.$$
\noindent Now, we consider \eqref{shell_model_integral_weak_form_pathwise_infinite_dimensional} in $H_{m}$,
\begin{equation}\label{shell_model_pathwise_finite_dimensional}
u^{(m)}(t) + \nu \int_{0}^{t}A(u^{(m)}(s))ds = u^{(m)}(0) + \int_{0}^{t}\pi_{m}B(u^{(m)}(s),u^{(m)}(s))ds + \pi_{m}\omega(t).
\end{equation}

\noindent \underline{\textbf{Step 3}}\\

\noindent Since \eqref{shell_model_pathwise_finite_dimensional} is a finite dimensional system it admits a local solution in $H_{m}$. [Recall that Cauchy-Peano Theorem states that the ODE $u'(t)=f(t,u(t))$ where $f$ is continuous on $D$ with $f:D\subset \mathbb{R}\times\mathbb{R}\rightarrow \mathbb{R}$ and $u(t_{0})=u_{0}$ with $(t_{0},u_{0})\in D$ has a local solution in an interval containing $t_{0}$.] In particular, here we use the fact that the operator $B$ is locally Lipschitz. \\


\noindent \underline{\textbf{Step 4}}\\

\noindent In order to prove the global existence of \eqref{shell_model_pathwise_finite_dimensional} in $H_{m},$ we state the following Lemma.

\begin{lemma}\label{semigroup_lemma}
For the auxiliary linear equation in $H_{m}$
\begin{equation}\label{z^{(m)}}
z^{(m)}(t)+\int_{0}^{t}\nu Az^{(m)}(s)ds=\pi_{m}\omega(t)
\end{equation}
\noindent there exists a unique global continuous solution $z^{(m)}$ in $H_{m}$, given by
\begin{equation}\nonumber
z^{(m)}(t) = \pi_{m}z(t)
\end{equation}
\noindent where $z\in C([0,T]:H)$ is given by
\begin{equation}\label{z}
z(t) = S(t)\omega (t) - \int_{0}^{t}\nu AS(t-s)(\omega(s)-\omega(t))ds
\end{equation}
\noindent where $S(t)$ is the analytic semigroup in $H$ generated by $\nu A$.
\end{lemma}
\begin{proof}
\noindent For a proof of Lemma \eqref{semigroup_lemma}, we refer to \cite{DaPrato}.
\end{proof}

\noindent \underline{\textbf{Step 5}}\\

\noindent We show that \eqref{shell_model_pathwise_finite_dimensional} has a global solution in $H_{m}$. Let $z^{(m)}(t)$ be defined as in Lemma \eqref{semigroup_lemma} and define
\begin{equation}\nonumber
v^{(m)}(t):=u^{(m)}(t)-z^{(m)}(t),\ t\in[0,T].
\end{equation}
\noindent Then, using \eqref{shell_model_pathwise_finite_dimensional} and \eqref{z^{(m)}} we get
\begin{equation}\nonumber
v^{(m)}(t) + \nu \int_{0}^{t}A(v^{(m)}(s))ds = v^{(m)}(0) + \int_{0}^{t}\pi_{m}B(v^{(m)}(s)+z^{(m)}(s),v^{(m)}(s)+z^{(m)}(s))ds
\end{equation}
\noindent which is differentiable in $t$. Thus we have
\begin{equation}\label{ODE}
\frac{d}{dt}v^{(m)}(t) + \nu Av^{(m)}(t) = \pi_{m}B(v^{(m)}(t)+z^{(m)}(t),v^{(m)}(t)+z^{(m)}(t)).
\end{equation}
\noindent We take the inner product $\langle \cdot , \cdot \rangle_{H}$ of equation \eqref{ODE} with $v^{(m)}(t)$. We have
\begin{align}\label{ODE_inner_product}
\bigg\langle \frac{dv^{(m)}(t)}{dt},v^{(m)}(t)\bigg\rangle_{H}  + \langle \nu Av^{(m)}(t),v^{(m)}(t)\rangle_{H} &= \langle \pi_{m}B(v^{(m)}(t)\nonumber
\\& +z^{(m)}(t),v^{(m)}(t)+z^{(m)}(t)),v^{(m)}(t)\rangle_{H}.
\end{align}
\noindent Consider the terms
\begin{align}\label{inner_product_1_first_term}
\bigg\langle \frac{dv^{(m)}(t)}{dt},v^{(m)}(t)\bigg\rangle_{H} &= \sum_{j=1}^{\infty}(v^{(m)}(t)')_{j}(v^{(m)}(t))_{j}\nonumber
\\& = \frac{1}{2}\sum_{j=1}^{\infty}2(v^{(m)}(t)')_{j}(v^{(m)}(t))_{j}\nonumber
\\& = \frac{1}{2}\sum_{j=1}^{\infty}\frac{d}{dt}(v^{(m)}(t))_{j}^{2}\nonumber
\\& = \frac{1}{2}\frac{d}{dt}\sum_{j=1}^{\infty}(v^{(m)}(t))_{j}^{2}\nonumber
\\& = \frac{1}{2}\frac{d}{dt}|(v^{(m)})(t)|_{H}^{2},
\end{align}

\begin{align}\label{inner_product_1_second_term}
\langle \nu Av^{(m)}(t),v^{(m)}(t)\rangle_{H} &= \nu\sum_{j=1}^{\infty}(Av^{(m)}(t))_{j}(v^{(m)}(t))_{j}\nonumber
\\& = \nu\sum_{j=1}^{\infty}k_{j}^{2}(v^{(m)}(t))_{j}(v^{(m)}(t))_{j}\nonumber
\\& = \nu\sum_{j=1}^{\infty}k_{j}^{2}(v^{(m)}(t))_{j}^{2}\nonumber
\\& = \nu|v^{(m)}(t)|^{2}_{V}
\end{align}
\noindent and
\begin{align}\label{inner_product_1_third_term}
& \langle \pi_{m}B(v^{(m)}(t)+z^{(m)}(t),v^{(m)}(t)+z^{(m)}(t)),v^{(m)}(t)\rangle_{H}\nonumber
\\& = \langle [\pi_{m}B(v^{(m)}(t),v^{(m)}(t)+z^{(m)}(t))+\pi_{m}B(z^{(m)}(t),v^{(m)}(t)+z^{(m)}(t))],v^{(m)}(t)\rangle_{H} \nonumber
\\& = \langle [\pi_{m}B(v^{(m)}(t),v^{(m)}(t))+\pi_{m}B(v^{(m)}(t),z^{(m)}(t))+\pi_{m}B(z^{(m)}(t),v^{(m)}(t))\nonumber
\\& +\pi_{m}B(z^{(m)}(t),z^{(m)}(t))],v^{(m)}(t)\rangle_{H} \nonumber
\\& = \langle \pi_{m}B(v^{(m)}(t),z^{(m)}(t)),v^{(m)}(t)\rangle_{H}+\langle \pi_{m}B(z^{(m)}(t),z^{(m)}(t)),v^{(m)}(t)\rangle_{H}
\end{align}
\noindent since $\langle \pi_{m}B(v^{(m)}(t),v^{(m)}(t),v^{(m)}(t)\rangle_{H} = \langle \pi_{m}B(z^{(m)}(t),v^{(m)}(t),v^{(m)}(t)\rangle_{H} = 0.$\\

\noindent Consider the first term in \eqref{inner_product_1_third_term}. Then, using Young's inequality we have
\begin{align}\label{inner_product_1_third_term_part_1}
\langle \pi_{m}B(v^{(m)}(t),z^{(m)}(t)),v^{(m)}(t)\rangle_{H}
&= \sum_{j=1}^{\infty}(\pi_{m}B(v^{(m)}(t),z^{(m)}(t)))_{j}(v^{(m)}(t))_{j}\nonumber
\\&= \sum_{j=1}^{\infty}k_{j}^{-1}(\pi_{m}B(v^{(m)}(t),z^{(m)}(t)))_{j}k_{j}(v^{(m)}(t))_{j}\nonumber
\\&\leq \sum_{j=1}^{\infty}\bigg[\frac{1}{\nu/4}k_{j}^{-2}(\pi_{m}B(v^{(m)}(t),z^{(m)}(t)))_{j}^{2}+\frac{\nu}{4}k_{j}^{2}(v^{(m)}(t))_{j}^{2}\bigg]\nonumber
\\&=\frac{1}{\nu/4}\sum_{j=1}^{\infty}k_{j}^{-2}(\pi_{m}B(v^{(m)}(t),z^{(m)}(t)))_{j}^{2}+\frac{\nu}{4}\sum_{j=1}^{\infty}k_{j}^{2}(v^{(m)}(t))_{j}^{2}\nonumber
\\&=\frac{1}{\nu/4}|\pi_{m}B(v^{(m)}(t),z^{(m)}(t))|^{2}_{V'}+\frac{\nu}{4}|v^{(m)}(t)|^{2}_{V}\nonumber
\\&\leq\frac{1}{\nu/4}|B(v^{(m)}(t),z^{(m)}(t))|^{2}_{V'}+\frac{\nu}{4}|v^{(m)}(t)|^{2}_{V}\nonumber
\\&\leq C_{\nu}|v^{(m)}(t)|^{2}_{H}|z^{(m)}(t))|^{2}_{H}+\frac{\nu}{4}|v^{(m)}(t)|^{2}_{V}.
\end{align}
\noindent Again, using Young's inequality for the second term in \eqref{inner_product_1_third_term}
\begin{equation}\label{inner_product_1_third_term_part_2}
\langle \pi_{m}B(z^{(m)}(t),z^{(m)}(t)),v^{(m)}(t)\rangle_{H} \leq C_{\nu}|z^{(m)}(t)|^{4}_{H}+\frac{\nu}{4}|v^{(m)}(t)|^{2}_{V}.
\end{equation}
\noindent Thus from \eqref{inner_product_1_third_term_part_1} and \eqref{inner_product_1_third_term_part_2} we have the estimate for the non-linear term
\begin{equation}\label{inner_product_1_third_term_estimate}
\langle \pi_{m}B(v^{(m)}(t),z^{(m)}(t)),v^{(m)}(t)\rangle_{H} \leq C_{\nu}(|v^{(m)}(t)|^{2}_{H}|z^{(m)}(t))|^{2}_{H}+|z^{(m)}(t)|^{4}_{H})+\frac{\nu}{2}|v^{(m)}(t)|^{2}_{V}.
\end{equation}
\noindent From \eqref{ODE_inner_product}, \eqref{inner_product_1_first_term}, \eqref{inner_product_1_second_term} and \eqref{inner_product_1_third_term_estimate} we have
\begin{equation}\label{ODE_difference_bound_1}
\frac{1}{2}\frac{d}{dt}|(v^{(m)})(t)|_{H}^{2} + \nu|v^{(m)}(t)|^{2}_{V} \leq C_{\nu}(|v^{(m)}(t)|^{2}_{H}|z^{(m)}(t))|^{2}_{H}+|z^{(m)}(t)|^{4}_{H})+\frac{\nu}{2}|v^{(m)}(t)|^{2}_{V}.
\end{equation}
\begin{equation}\label{ODE_difference_bound}
\text{Thus }\frac{1}{2}\frac{d}{dt}|v^{(m)}(t)|_{H}^{2} + \frac{\nu}{2}|v^{(m)}(t)|^{2}_{V} \leq C_{\nu}(|v^{(m)}(t)|^{2}_{H}|z^{(m)}(t))|^{2}_{H}+|z^{(m)}(t)|^{4}_{H}).
\end{equation}
\noindent 
Integrating from $0$ to $t$, we obtain
\begin{equation}\nonumber
|v^{(m)}(t)|_{H}^{2} \leq |v^{(m)}(0)|_{H}^{2} + C_{\nu}\int_{0}^{t}|z^{(m)}(s)|^{4}_{H}ds + C_{\nu}\int_{0}^{t}|v^{(m)}(s)|^{2}_{H}|z^{(m)}(s)|^{2}_{H}ds.
\end{equation}
\noindent Applying the Gronwall Lemma we get
\begin{align}
& \sup_{0\leq t \leq T}|v^{(m)}(t)|^{2}_{H}\nonumber
\\ & \leq |v^{(m)}(0)|^{2}_{H}\exp\bigg(\int_{0}^{T}|z^{(m)}(s)|^{2}_{H}ds\bigg)+C_{\nu}\int_{0}^{T}|z^{(m)}(s)|^{4}_{H}\exp\bigg(\int_{s}^{T}|z^{(m)}(r)|^{2}_{H}dr\bigg)ds.\nonumber
\end{align}
\noindent Since $z^{(m)}(0)=0$, then $v^{(m)}(0)=\pi_{m}u(0)$. However, $z^{(m)}\in C([0,T]:H)$, thus, $\displaystyle\sup_{t\in[0,T]}|z^{(m)}(t)|^{2}_{H}$ depends only on $|\omega|_{C^{\alpha}([0,T]:H)}$. Thus we have
\begin{equation}\label{sup_1}
\sup_{t\in[0,T]}|v^{(m)}(t)|^{2}_{H}\leq C(\nu,T,|u(0)|_{H},|\omega|_{C^{\alpha}([0,T]:H)}).
\end{equation}
\noindent On the other hand $$|u^{(m)}(t)|_{H} \leq |v^{(m)}(t)|_{H} + |z^{(m)}(t)|_{H}.$$ Thus,
\begin{equation}\nonumber
\sup_{t\in[0,T]}|u^{(m)}(t)|^{2}_{H}\leq C(\nu,T,|u(0)|_{H},|\omega|_{C^{\alpha}([0,T]:H)}).
\end{equation}
\noindent This a priori estimate gives the existence of $u^{(m)}$ on $[0,T]$ for arbitrary $T\geq 0$. Thus, we have the global existence for finite dimensional approximation $u^{(m)}$ satisfying the equation \eqref{shell_model_pathwise_finite_dimensional}.\\

\noindent \underline{\textbf{Step 6}}\\

\noindent We now prove that there exists $u\in C([0,T]:H)$ solving the equation \eqref{shell_model_integral_weak_form_pathwise_infinite_dimensional}, provided $u(0)\in H$ and for $\omega\in C^{\alpha}([0,T]:H)$ for some $0\leq\alpha<\frac{1}{2}$. Integrating \eqref{ODE_difference_bound} from $0$ to $t$, we get
\begin{equation}\nonumber
\int_{0}^{t}|v^{(m)}(s)|^{2}_{V}ds \leq C_{\nu}\int_{0}^{t}(|v^{(m)}(s)|^{2}_{H}|z^{(m)}(s))|^{2}_{H}+|z^{(m)}(s)|^{4}_{H})ds.
\end{equation}
\noindent Now using the estimate \eqref{sup_1} we have
\begin{equation}\nonumber
\int_{0}^{t}|v^{(m)}(s)|^{2}_{V}ds \leq C_{\nu}\bigg(\sup_{t\in [0,T]}|v^{(m)}(t)|^{2}_{H}\sup_{t\in [0,T]}|z^{(m)}(t)|^{2}_{H}+\sup_{t\in [0,T]}|z^{(m)}(t)|^{4}_{H}\bigg)T.
\end{equation}
\noindent Hence,
\begin{equation}\label{difference_V_estimate}
\int_{0}^{T}|v^{(m)}(t)|^{2}_{V}dt\leq C(\nu,T,|u(0)|_{H},|\omega|_{C^{\alpha}([0,T]:H)}).
\end{equation}
\noindent Consider the equation \eqref{ODE}. We have
\begin{align}\label{difference_derivative_estimate}
\bigg|\frac{dv^{(m)}(t)}{dt}\bigg|_{V'}^{2}
&= \sum_{j=1}^{\infty}k_{j}^{-2}\bigg(\frac{dv^{(m)}(t)}{dt}\bigg)^{2}_{j}\nonumber
\\&= \sum_{j=1}^{\infty}k_{j}^{-2}(-\nu Av^{(m)}(t) + \pi_{m}B(v^{(m)}(t)+z^{(m)}(t),v^{(m)}(t)+z^{(m)}(t)))^{2}_{j}\nonumber
\\&= \sum_{j=1}^{\infty}k_{j}^{-2}(-\nu Av^{(m)}(t) + \pi_{m}B(u^{(m)}(t),u^{(m)}(t))^{2})_{j}\nonumber
\\&\leq C_{\nu}\sum_{j=1}^{\infty}k_{j}^{-2}(Av^{(m)}(t))^{2}_{j}+C\sum_{j=1}^{\infty}k_{j}^{-2}(\pi_{m}B(u^{(m)}(t),u^{(m)}(t)))^{2}_{j}\nonumber
\\&= C_{\nu}\sum_{j=1}^{\infty}k_{j}^{-2}k_{j}^{4}(v^{(m)}(t))^{2}_{j}+C|\pi_{m}B(u^{(m)},u^{(m)})|_{V'}^{2}\nonumber
\\&\leq C_{\nu}\sum_{j=1}^{\infty}k_{j}^{2}(v^{(m)}(t))^{2}_{j}+C|B(u^{(m)},u^{(m)})|_{V'}^{2}\nonumber
\\& \leq C_{\nu} |v^{(m)}|^{2}_{V}+C|u^{(m)}|^{4}_{H}.
\end{align}
\noindent From \eqref{difference_V_estimate} and \eqref{difference_derivative_estimate} we get
\begin{equation}\label{difference_W12_bound}\nonumber
|v^{(m)}|_{W^{1,2}(0,T;V')}\leq C(\nu,T,|u(0)|_{H},|\omega|_{C^{\alpha}([0,T]:H)}).
\end{equation}

\noindent Thus the sequence $\{v^{(m)}\}$ is bounded in $L^{2}(0,T:V)$ and in $W^{1,2}(0,T:V')$. Hence, we can use the compactness theorem \ref{Temam_theorem_3} (see remark \ref{Temam_compactness_theorems} and also see \cite{JLLions_1969}) and get a subsequence $\{v^{(m_{q})}\}$ which converges strongly to some $v$ in $L^{2}(0,T:H)$. Since $u^{(m)}(t)=v^{(m)}(t)+z^{(m)}(t)$ we have $\{u^{(m_{q})}\}$ converges strongly to some $u$ in the same topology. Thus, we can pass to the limit, and we have proved that a solution exists for equation \eqref{shell_model_integral_weak_form_pathwise_infinite_dimensional} for all $\phi\in D(A)$.

\noindent Since the sequence $\{v^{(m)}\}$ is bounded in $L^{2}(0,T:V)$ and in $W^{1,2}(0,T:V')$, it follows that $v$ is in $L^{2}(0,T:V)$ and $W^{1,2}(0,T:V')$, hence in $C([0,T],H)$ by Theorem \ref{Temam_theorem_1} (see remark \ref{Temam_compactness_theorems}). Thus, $u\in C([0,T],H)$.\\

\begin{remark}\label{Temam_compactness_theorems}
Recall the following compactness theorems.
\begin{theorem}\label{Temam_theorem_1}
Let $V,$ $H,$ $V'$ be three Hilbert spaces such that $V\subset H\subset V'$, $V'$ being the dual space of $V.$ If a function $u$ belongs to $L^{2}(0,T;V)$ and its derivative $u'$ belongs to $L^{2}(0,T;V')$, then $u$ is almost everywhere equal to a function in $C([0,T];H)$ and we have the following equality which holds in the scalar distribution sense on $(0,T)$: 
\begin{equation}\label{derivative_equality}
\frac{d}{dt}|u|^{2}_{H}=2\langle u',u\rangle_{H}.
\end{equation}
This equality \eqref{derivative_equality} is meaningful since the functions $$t\rightarrow |u(t)|^{2}_{H},\ t\rightarrow \langle u'(t),u(t)\rangle$$ are both integrable on $[0,T].$
\end{theorem}
\begin{proof}
See \cite{Temam_NSE_Numerical_Analysis}, page 260.
\end{proof}
\begin{theorem}\label{Temam_theorem_2}
Let $X$ and $Y$ be two Banach spaces such that $X\subset Y$ with a continuous injection. If a function $\phi\in L^{\infty}(0,T;X)$ and is weakly continuous with values in $Y$, then $\phi$ is weakly continuous with values in $X$. 
\end{theorem}
\begin{proof}
See \cite{Temam_NSE_Numerical_Analysis}, page 263.
\end{proof}

\noindent Consider the following assumptions. \\
\noindent A1. Let $X_{0}$, $X$, $X_{1}$ be three Banach spaces such that $ X_{0}\subset X \subset X_{1},$ where the injections are continuous and $X_{1} \text{ is reflexive } (i=0,1),$ $\text{ the injection } X_{0}\rightarrow X \text{ is compact. }$\\
\noindent A2. Let $T>0$ be a fixed finite number and let $\alpha_{0}$, $\alpha_{1}$ be two finite numbers such that $\alpha_{i}>1$, $(i=0,1)$. We consider the space
$$Y=Y(0,T;\alpha_{0}, \alpha_{1};X_{0}, X_{1})$$
$$Y=\displaystyle\{v\in L^{\alpha_{0}}(0,T;X_{0}),\ v'=\frac{dv}{dt}\in L^{\alpha_{1}}(0,T;X_{1})\displaystyle\}.$$
The space $Y$ is provided with the norm 
$$\|v\|_{Y}=\|v\|_{L^{\alpha_{0}}(0,T;X_{0})}+\|v\|_{L^{\alpha_{1}}(0,T;X_{1})},$$
which makes it a Banach space. It is evident that $$Y\subset L^{\alpha_{0}}(0,T;X),$$
which is a continuous injection. 
\begin{theorem}\label{Temam_theorem_3}
Under the above assumptions A1 and A2 the injection of $Y$ into $L^{\alpha_{0}}(0,T;X)$ is compact.
\end{theorem}
\begin{proof}
See \cite{Temam_NSE_Numerical_Analysis}, page 271.
\end{proof}
\end{remark}

\noindent \underline{\textbf{Step 7}}\\

\noindent In order to prove that there exists a solution to the stochastic equation \eqref{shell_model_differential_form_infinite_dimensional}, we will have to show that the solution of \eqref{shell_model_integral_weak_form_pathwise_infinite_dimensional} depends continuously on $\omega$.

\noindent Suppose $u(0)\in H$, $\omega^{(1)}, \omega^{(2)} \in C^{\alpha}([0,T]:H)$ for some $0\leq\alpha<\frac{1}{2}$. Let $u^{(1)}$ and $u^{(2)}$ be solutions of \eqref{shell_model_integral_weak_form_pathwise_infinite_dimensional} corresponding to $\omega^{(1)}$ and $\omega^{(2)}$ respectively. Let $z^{(1)}$ and $z^{(2)}$ be as in \eqref{z} corresponding to $\omega^{(1)}$ and $\omega^{(2)}$ respectively. We define
\begin{equation}\nonumber
v^{(1)}=u^{(1)}-z^{(1)},\ v^{(2)}=u^{(2)}-z^{(2)}\text{ and}
\end{equation}
\begin{equation}\nonumber
y=v^{(1)}-v^{(2)}=u^{(1)}-u^{(2)}-(z^{(1)}-z^{(2)}).
\end{equation}
\noindent We have
\begin{equation}\nonumber
\frac{1}{2}\frac{d}{dt}|y|^{2}_{H}+\nu|y|^{2}_{V}\leq \frac{\nu}{2}|y|^{2}_{V}+C_{\nu}|u^{(1)}|^{2}_{H}|y|^{2}_{H} +C_{\nu}(|u^{(1)}|_{H}+|u^{(2)}|_{H})^{2}|z^{(1)}-z^{(2)}|^{2}_{H}.
\end{equation}
\noindent Using the same techniques used for the existence, we have the following estimate,
\begin{align}
& \sup_{t\in[0,T]}|u^{(1)}(t)-u^{(2)}(t)|_{H}\nonumber
\\& \leq C\bigg(\nu,T,|u(0)|_{H},\int_{0}^{T}|u^{(1)}(r)|^{2}_{H}dr,\int_{0}^{T}|u^{(2)}(r)|^{2}_{H}dr\bigg)|\omega^{(1)}-\omega^{(2)}|_{C^{\alpha}([0,T]:H)}.\nonumber
\end{align}
\noindent Consider equation \eqref{shell_model_differential_form_infinite_dimensional}. For $P$-a.e. $\omega\in\Omega$ the function $t\mapsto W_{t}(\omega)$ is in $C^{\alpha}([0,T]:H)$ for $0\leq\alpha<\frac{1}{2}$. Let $\alpha\in[0,\frac{1}{2})$. Choose $\Omega=C^{\alpha}([0,T]:H)$ with Borel $\sigma$-algebra, filteration associated to the canonical process $t\mapsto \omega(t)$ and the measure $P$ given by the law of the previous Brownian motion $\{W(t)\}_{t\geq 0}$. Thus the canonical process $t\mapsto \omega(t)$ is a Brownian motion in $H$ with covariance $Q$ with all paths in $C^{\alpha}([0,T]:H)$. Suppose $u(0):\Omega\rightarrow H$ is a given $\mathcal{F}_{0}$-measurable random variable. Suppose $u(t,\omega)$ is a solution of equation \eqref{shell_model_integral_weak_form_pathwise_infinite_dimensional} for a given path $t\mapsto \omega(t)$ of the Brownian motion described above with the given initial condition. Since $u(\cdot,\omega)$ depends on $\omega$ continuously the function $(t,\omega)\mapsto u(t,\omega)$ is a $\mathcal{F}$-measurable stochastic process for all $t\geq0$. It is also adapted since the same argument is true for any interval $[0,t]$ with arbitrary $t>0$. Hence, given a $\mathcal{F}_{0}$-measurable random variable $u(0):\Omega\rightarrow H$, we have proved that there exists an adapted process $\{u(t)\}_{t\geq 0}$ that solves equation \eqref{shell_model_differential_form_infinite_dimensional}.

\end{proof}

\subsubsection{Uniqueness of the Solution}

\begin{theorem}\label{uniqueness_of_the_solution}
Suppose $(u_{(1)})_{t\geq 0}$ and $(u_{(2)})_{t\geq 0}$ are two solutions of \eqref{shell_model_differential_form_infinite_dimensional} in $H$ with the
associated initial conditions $u_{(1)}(0)$ and $u_{(2)}(0)$ respectively. Then there is a constant $C_{\nu}$ such that
\begin{equation}\nonumber
|u_{(1)}-u_{(2)}|^{2}_{H} \leq \bigg(\exp\bigg[C_{\nu}\int_{0}^{t}\bigg(|u_{(1)}|^{2}_{H} +
|u_{(2)}|^{2}_{H}\bigg)ds\bigg]\bigg)|u_{(1)}(0)-u_{(2)}(0)|^{2}_{H}\nonumber
\end{equation}
for all $t \geq 0$ with probability $1$. If $$u_{(1)}(0)=u_{(2)}(0)$$ then $$\mathcal{P}[u_{(1)} = u_{(2)} \text{ for all } t\geq 0] = 1.$$
\end{theorem}
\begin{proof}

We will show the pathwise uniqueness. Let $m\in \mathbb{N}$.  Let $v=u_{(1)}-u_{(2)}$. Then
\begin{equation}\nonumber
v^{(m)} + \nu\int_{0}^{t}Av^{(m)}ds = v^{(m)}(0) + \int_{0}^{t}[\pi_{m}B(u^{(1)},v)-\pi_{m}B(v,u^{(2)})]ds
\end{equation}
\noindent Then $v^{(m)}$ is differentiable it $t$ and
\begin{equation}\nonumber
\frac{d}{dt}v^{(m)} + \nu Av^{(m)} = \pi_{m}B(u_{(1)},v) + \pi_{m}B(v,u_{(2)}).
\end{equation}
\noindent Now consider
\begin{equation}\nonumber
\langle v^{(m)'},v^{(m)}\rangle +\nu\langle Av^{(m)},v^{(m)}\rangle = \langle \pi_{m}B(u_{(1)},v),v^{(m)}\rangle + \langle \pi_{m}B(v,u_{(2)}),v^{(m)}\rangle
\end{equation}
\noindent Since
\begin{equation}\nonumber
\langle v^{(m)'},v^{(m)}\rangle = \frac{1}{2}\frac{d}{dt}|v^{(m)}|^{2}_{H},
\end{equation}
\begin{equation}\nonumber
\langle Av^{(m)},v^{(m)}\rangle = |v^{(m)}|^{2}_{V},
\end{equation}
\begin{equation}\nonumber
\langle \pi_{m}B(u_{(1)},v),v^{(m)}\rangle \leq \frac{\nu}{4}|v^{(m)}|^{2}_{V} + C_{\nu}|u_{(1)}|^{2}_{H}|v|^{2}_{H},
\end{equation}
\noindent and
\begin{equation}\nonumber
\langle \pi_{m}B(v,u_{(2)}),v^{(m)}\rangle \leq \frac{\nu}{4}|v^{(m)}|^{2}_{V} + C_{\nu}|u_{(2)}|^{2}_{H}|v|^{2}_{H},
\end{equation}
\noindent we get
\begin{equation}\nonumber
\frac{d}{dt}|v^{(m)}|^{2}_{H} \leq C_{\nu}\int_{0}^{t}\bigg(|u_{(1)}|_{H}^{2} + |u_{(2)}|_{H}^{2}\bigg)|v|_{H}^{2}ds.
\end{equation}
\noindent Thus
\begin{equation}\nonumber
|v^{(m)}|^{2}_{H} \leq |v^{(m)}(0)|^{2}_{H} +C_{\nu}\int_{0}^{t}\bigg(|u_{(1)}|_{H}^{2} + |u_{(2)}|_{H}^{2}\bigg)|v|_{H}^{2}ds.
\end{equation}
\noindent As $m\rightarrow \infty$ we get
\begin{equation}\nonumber
|v|^{2}_{H} \leq |v(0)|^{2}_{H} +C_{\nu}\int_{0}^{t}\bigg(|u_{(1)}|_{H}^{2} + |u_{(2)}|_{H}^{2}\bigg)|v|_{H}^{2}ds.
\end{equation}
\noindent Hence the result follows from Gronwall's Lemma.
\end{proof}
\subsection{An Estimate on the Expectation}
\noindent We will prove the following theorem in order to get an estimate on $\mathbb{E}|u(t)|_{H}^{2}$.
\begin{theorem}\label{estimate_for_expectation}
Let $u(0):\Omega \rightarrow H$ be a $\mathcal{F}_{0}$-measurable random variable and $(u(t))_{t\geq 0}$ be the unique continuous adapted process in $H$ solving the equation \eqref{shell_model_differential_form_infinite_dimensional}. If $\mathbb{E}|u(0)|_{H}^{2}<\infty$ then
\begin{equation}\nonumber
\mathbb{E}|u(t)|_{H}^{2} +2\nu\int_{0}^{t} \mathbb{E}|u(s)|^{2}_{V}ds=\mathbb{E}|u(0)|_{H}^{2}+t\cdot \text{Tr}Q \ \text{ for all } t\in [0,T]
\end{equation}
and
\begin{equation}\nonumber
\mathbb{E}\bigg[\sup_{t\in[0,T]}|u(t)|^{2}_{H}+\nu\int_{0}^{t}|u(s)|^{2}_{V}ds\bigg]\leq C(\mathbb{E}|u(0)|_{H}^{2}, T, \text{Tr}Q).
\end{equation}
\end{theorem}
\noindent Further if $\mathbb{E}|u(0)|_{H}^{p}<\infty$ for $p\geq 2$ then
\begin{equation}\nonumber
\mathbb{E}\bigg[\sup_{t\in[0,T]}|u(t)|^{p}_{H}\bigg]\leq C(\mathbb{E}|u(0)|_{H}^{p}, T, \text{Tr}Q, p).
\end{equation}
and
\begin{equation}\nonumber
\frac{1}{T}\mathbb{E}\bigg[\int_{0}^{T}|u(s)|^{p}_{H}ds\bigg]\leq C(\text{Tr}Q, \nu, p, k_{0})\bigg[1+\frac{\mathbb{E}|u(0)|^{p}_{H}}{T}\bigg].
\end{equation}

\begin{proof}
Let $$\tau_{R}:=\inf\{t\geq 0 : |u(t)|^{2}_{H}\geq R\} \wedge T$$ be a stopping time. We have
\begin{align}
u(t\wedge \tau_{R})&=u(0)+\int_{0}^{t\wedge \tau_{R}}[-\nu Au(s)-B(u(s),u(s))]ds +W(t\wedge \tau_{R})\nonumber
\end{align}
From the It$\hat{o}$ formula, we get
\begin{align}
|u(t\wedge \tau_{R})|^{p}_{H}&=|u(0)|^{p}_{H}+\int_{0}^{t\wedge \tau_{R}}p\cdot \langle |u(s)|^{p-1}_{H},dW(s)\rangle_{H} \nonumber
\\& + \int_{0}^{t\wedge \tau_{R}}\langle  p|u(s)|^{p-1}_{H},[-\nu Au(s)-B(u(s),u(s))]\rangle_{H} ds \nonumber
\\& + \frac{1}{2}p(p-1)\text{Tr}(Q)\int_{0}^{t\wedge \tau_{R}}|u(s)|^{p-2}_{H} ds \nonumber
\end{align}
Thus,
\begin{align}
& \sup_{t\in[0,r]}|u(t\wedge \tau_{R})|^{p}_{H}+p\nu\int_{0}^{r\wedge \tau_{R}}|u(s)|^{p-2}_{H}\langle Au(s),u(s)\rangle_{H}  ds\nonumber
\\ &\leq 2|u(0)|^{p}_{H}+2p\nu\int_{0}^{r\wedge \tau_{R}}|u(s)|^{p-2}_{H} \langle B(u(s),u(s)),u(s) \rangle_{H}  ds\nonumber
\\ & + 2\sup_{t\in[0,r]}|M_{t}^{(p)}|+p(p-1)\text{Tr}(Q)\int_{0}^{r\wedge \tau_{R}}|u(s)|^{p-2}_{H} ds \nonumber
\end{align}
where $$ M_{t}^{(p)}=p\int_{0}^{t\wedge \tau_{R}}|u(s)|^{p-2}_{H}\langle u(s),dW(s)\rangle_{H}ds.$$
Using the Burkholder-Davis-Gundy inequality, we see that
\begin{align}
\mathbb{E}\sup_{t\in[0,r]}|M_{t}^{(p)}|&\leq C(p,\text{Tr}Q)\mathbb{E}\bigg[\bigg( \int_{0}^{r\wedge \tau_{R}}|u(s)|^{2p-2}_{H} ds\bigg)^{1/2}\bigg]\nonumber
\\ & \leq \frac{1}{4}\mathbb{E}\bigg[ \sup_{t\in[0,r]}|u(t\wedge \tau_{R})|^{p}_{H} \bigg] +C'(p,\text{Tr}Q)\mathbb{E}\int_{0}^{r\wedge \tau_{R}}|u(s)|^{p-2}_{H} ds.\nonumber
\end{align}
Thus,
\begin{align}
& \frac{1}{2}\mathbb{E}\sup_{t\in[0,r]}|u(t\wedge \tau_{R})|^{p}_{H} + p\nu \mathbb{E}\int_{0}^{r\wedge \tau_{R}}|u(s)|^{p-2}_{H}\langle Au(s),u(s)\rangle_{H} ds\nonumber
\\ &\leq 2\mathbb{E}|u(0)|^{p}_{H} + 2p\nu \mathbb{E} \int_{0}^{r\wedge \tau_{R}}|u(s)|^{p-2}_{H} \langle B(u(s),u(s)),u(s) \rangle_{H} ds\nonumber
\\ &+ 2C'(p,\text{Tr}Q)\mathbb{E}\int_{0}^{r\wedge \tau_{R}}|u(s)|^{p-2}_{H} ds + p(p-1)\text{Tr}(Q)\int_{0}^{r\wedge \tau_{R}}\mathbb{E}|u(s)|^{p-2}_{H} ds.\nonumber
\end{align}
That is,
\begin{align}\nonumber
\frac{1}{2}\mathbb{E}\sup_{t\in[0,r]}|u(t\wedge \tau_{R})|^{p}_{H}\leq \mathbb{E}|u(0)|^{p}_{H}+1+C''(p,\text{Tr}Q)\mathbb{E}\int_{0}^{r\wedge \tau_{R}} \sup_{t\in[0,r]}|u(t\wedge \tau_{R})|^{p}_{H}ds.
\end{align}
By applying the Gronwall Lemma we get
\begin{equation}\nonumber
\mathbb{E}\sup_{t\in[0,r]}|u(t\wedge \tau_{R})|^{p}_{H}\leq C''(p,\text{Tr}Q, \mathbb{E}|u(0)|^{p}_{H}).
\end{equation}
By applying monotone convergence theorem we obtain
\begin{equation}\nonumber
\mathbb{E}\sup_{t\in[0,r]}|u(t)|^{p}_{H}\leq C''(p,\text{Tr}Q, \mathbb{E}|u(0)|^{p}_{H}).
\end{equation}
Since
\begin{equation}\nonumber
k_{0}p\nu \mathbb{E}\int_{0}^{r\wedge \tau_{R}}|u(s)|^{p}_{H}ds\leq 2\mathbb{E}|u(0)|^{p}_{H}+2C''(p,\text{Tr}Q)\mathbb{E}\int_{0}^{r\wedge \tau_{R}}|u(s)|^{p-2}_{H}ds
\end{equation}
we have
\begin{equation}\nonumber
\frac{1}{2}k_{0}p\nu \mathbb{E}\int_{0}^{r\wedge \tau_{R}}|u(s)|^{p}_{H}ds\leq 2\mathbb{E}|u(0)|^{p}_{H}+rC(p,\text{Tr}Q,\nu,k_{0}).
\end{equation}
\end{proof}

\subsection[Regularity of Stochastic Shell Models of Turbulence] {Regularity of Stochastic Shell Models of Turbulence}

\subsubsection{Existence of Strong Solution}
\noindent We proved above that $u\in C([0,T],H)$ when the initial condition $u(0)\in C([0,T],H)$. Analogous to the work related to Navier-Stokes equations we will show that \eqref{shell_model_differential_form_infinite_dimensional} possesses more regular solutions as stated in the following theorem.
\begin{theorem}
Let $\{W(t)\}_{t\geq 0}$ be an $H$-valued Brownian motion defined on a filtered probability space $(\Omega, \mathcal{F}, (\mathcal{F})_{t\geq 0},P)$ with nuclear covariance operator $Q$ and let $T>0$.  We are given $u(0)$, a $\mathcal{F}_{0}$-measurable random variable as the initial condition. Then the unique solution of \eqref{shell_model_differential_form_infinite_dimensional} satisfies
\begin{equation}\nonumber
u\in L_{loc}^{\infty}((0,T],V)\cap L_{loc}^{2}((0,T],D(A))\cap L^{\infty}([0,T],H)\cap L^{2}([0,T],V)
\end{equation}
if the initial condition $u(0)\in H.$ The solution satisfies
\begin{equation}\nonumber
u \in C([0,T],V)\cap L^{2}([0,T],D(A))
\end{equation}
if the initial condition $u(0)\in V.$
\end{theorem}
\begin{proof}
\noindent To be rigorous we will perform our computations using a Galerkin approximation. We take the inner product $\langle \cdot , \cdot \rangle_{H}$ of equation \eqref{ODE} with $Av^{(m)}(t)$. We have
\begin{align}\label{ODE_inner_product_with_Av}
& \bigg\langle \frac{dv^{(m)}(t)}{dt},Av^{(m)}(t)\bigg\rangle_{H}  + \langle \nu Av^{(m)}(t),Av^{(m)}(t)\rangle_{H}\nonumber
\\& = \langle \pi_{m}B(v^{(m)}(t)+z^{(m)}(t),v^{(m)}(t)+z^{(m)}(t)),Av^{(m)}(t)\rangle_{H}.\nonumber
\end{align}

\noindent Consider the terms

\begin{align}
\bigg\langle \frac{dv^{(m)}(t)}{dt},Av^{(m)}(t)\bigg\rangle_{H}
&= \sum_{j=1}^{\infty}(v^{(m)}(t)')_{j}(Av^{(m)}(t))_{j}\nonumber
\\& = \frac{1}{2}\sum_{j=1}^{\infty}2(v^{(m)}(t)')_{j}k_{j}^{2}(v^{(m)}(t))_{j}\nonumber
\\& = \frac{1}{2}\sum_{j=1}^{\infty}\frac{d}{dt}k_{j}^{2}(v^{(m)}(t))_{j}^{2}\nonumber
\\& = \frac{1}{2}\frac{d}{dt}\sum_{j=1}^{\infty}k_{j}^{2}(v^{(m)}(t))_{j}^{2}\nonumber
\\& = \frac{1}{2}\frac{d}{dt}|(v^{(m)})(t)|_{V}^{2},\nonumber
\end{align}
\noindent and
\begin{align}
& \langle \pi_{m}B(v^{(m)}(t)+z^{(m)}(t),v^{(m)}(t)+z^{(m)}(t)),Av^{(m)}(t)\rangle_{H}\nonumber
\\&\leq |\pi_{m}B(v^{(m)}(t)+z^{(m)}(t),v^{(m)}(t)+z^{(m)}(t))|_{H}|Av^{(m)}(t)|_{H}\nonumber
\\& \leq C |v^{(m)}(t)+z^{(m)}(t)|_{H}|v^{(m)}(t)+z^{(m)}(t)|_{H}|Av^{(m)}(t)|_{H}\nonumber
\\& \leq \frac{C^{2}}{2\nu}|v^{(m)}(t)+z^{(m)}(t)|_{H}^{4}+\frac{\nu}{2}|Av^{(m)}(t)|_{H}^{2}.\nonumber
\end{align}
\noindent We have
\begin{equation}\nonumber
\frac{1}{2}\frac{d}{dt}|v^{(m)}(t)|_{V}^{2}+\frac{\nu}{2}|Av^{(m)}(t)|_{H}^{2} \leq \frac{C^{2}}{2\nu}|v^{(m)}(t)+z^{(m)}(t)|_{H}^{4}.
\end{equation}
\noindent Observe that
\begin{align}
& |v^{(m)}(t)+z^{(m)}(t)|_{H}^{4}\nonumber
\\& \leq |v^{(m)}(t)|_{H}^{4} + 4|v^{(m)}(t)|_{H}^{3}|z^{(m)}(t)|_{H} + 6|v^{(m)}(t)|_{H}^{2}||z^{(m)}(t)|_{H}^{2} + 4|v^{(m)}(t)|_{H}|z^{(m)}(t)|_{H}^{3}\nonumber
\\& \ \ \ \ + |z^{(m)}(t)|_{H}^{4}\nonumber
\\& \leq |v^{(m)}(t)|_{H}^{2}(|v^{(m)}(t)|_{H}^{2} + 4|v^{(m)}(t)|_{H}|z^{(m)}(t)|_{H} + 6|z^{(m)}(t)|_{H}^{2}) + 4|v^{(m)}(t)|_{H}|z^{(m)}(t)|_{H}^{3}\nonumber
\\& \ \ \ \ + |z^{(m)}(t)|_{H}^{4}\nonumber
\\& \leq |v^{(m)}(t)|_{H}^{2}(|v^{(m)}(t)|_{H}^{2} + 4(\frac{1}{2}|v^{(m)}(t)|_{H}^{2} + \frac{1}{2}|z^{(m)}(t)|_{H}^{2}) + 6|z^{(m)}(t)|_{H}^{2})\nonumber
\\& \ \ \ \ + 4(\frac{1}{2}|v^{(m)}(t)|_{H}^{2} + \frac{1}{2}|z^{(m)}(t)|_{H}^{2})|z^{(m)}(t)|_{H}^{2} + |z^{(m)}(t)|_{H}^{4}\nonumber
\\& \leq |v^{(m)}(t)|_{H}^{2}(3|v^{(m)}(t)|_{H}^{2} + 10|z^{(m)}(t)|_{H}^{2}) + 3|z^{(m)}(t)|_{H}^{4}\nonumber
\\& \leq k_{1}^{-2}|v^{(m)}(t)|_{V}^{2}(3|v^{(m)}(t)|_{H}^{2} + 10|z^{(m)}(t)|_{H}^{2}) + 3|z^{(m)}(t)|_{H}^{4}\nonumber
\end{align}
\noindent since
\begin{align}
& k_{1}^{2}|v^{(m)}(t)|_{H}^{2} = k_{1}^{2}\sum_{j=1}^{\infty}(v^{(m)}(t))^{2}_{j} = k_{1}^{2}(v^{(m)}(t))^{2}_{1}+k_{1}^{2}(v^{(m)}(t))^{2}_{2}+k_{1}^{2}(v^{(m)}(t))^{2}_{3}+\ldots \nonumber
\\& \leq k_{1}^{2}(v^{(m)}(t))^{2}_{1}+k_{2}^{2}(v^{(m)}(t))^{2}_{2}+k_{3}^{2}(v^{(m)}(t))^{2}_{3}+\ldots = \sum_{j=1}^{\infty}k_{j}^{2}(v^{(m)}(t))^{2}_{j} = |v^{(m)}(t)|_{V}^{2}.\nonumber
\end{align}
\noindent Thus, we have
\begin{equation}\nonumber
\frac{1}{2}\frac{d}{dt}|v^{(m)}(t)|_{V}^{2}+\frac{\nu}{2}|Av^{(m)}(t)|_{H}^{2} \leq \frac{C^{2}}{2k_{1}^{2}\nu}(3|v^{(m)}(t)|_{H}^{2} + 10|z^{(m)}(t)|_{H}^{2})|v^{(m)}(t)|_{V}^{2} + \frac{3C^{2}}{2k_{1}^{2}\nu}|z^{(m)}(t)|_{H}^{4}.
\end{equation}
\noindent Thus,
\begin{equation}\nonumber
\frac{d}{dt}|v^{(m)}(t)|_{V}^{2} - \frac{C^{2}}{2k_{1}^{2}\nu}(3|v^{(m)}(t)|_{H}^{2}+10|z^{(m)}(t)|_{H}^{2})|v^{(m)}(t)|_{V}^{2}
 \leq  \frac{3C^{2}}{2k_{1}^{2}\nu}|z^{(m)}(t)|_{H}^{4}.
\end{equation}
\noindent Let $0\leq t_{0}\leq t$. Then
\begin{align}
& \frac{d}{dt}|v^{(m)}(t)|_{V}^{2} e^{\Big(-\frac{C^{2}}{2k_{1}^{2}\nu}\int_{t_{0}}^{t}(3|v^{(m)}(t)|_{H}^{2}+10|z^{(m)}(t)|_{H}^{2})ds\Big)} \nonumber
\\& \ -\frac{C^{2}}{2k_{1}^{2}\nu} |v^{(m)}(t)|_{V}^{2}(3|v^{(m)}(t)|_{H}^{2}+10|z^{(m)}(t)|_{H}^{2})e^{\Big(-\frac{C^{2}}{2k_{1}^{2}\nu}\int_{t_{0}}^{t}(3|v^{(m)}(t)|_{H}^{2}+10|z^{(m)}(t)|_{H}^{2})ds\Big)}\nonumber
\\& \leq \frac{3C^{2}}{2k_{1}^{2}\nu}|z^{(m)}(t)|_{H}^{4} e^{\Big(-\frac{C^{2}}{2k_{1}^{2}\nu}\int_{t_{0}}^{t}(3|v^{(m)}(t)|_{H}^{2}+10|z^{(m)}(t)|_{H}^{2})ds\Big)}.\nonumber
\end{align}
\noindent Since $e^{\Big(-\frac{C^{2}}{2k_{1}^{2}\nu}\int_{t_{0}}^{t}(3|v^{(m)}(t)|_{H}^{2}+10|z^{(m)}(t)|_{H}^{2})ds\Big)}\leq 1,$
\begin{equation}\nonumber
\frac{d}{dt}\bigg[|v^{(m)}(t)|_{V}^{2}e^{\Big(-\frac{C^{2}}{2k_{1}^{2}\nu}\int_{t_{0}}^{t}(3|v^{(m)}(t)|_{H}^{2}+10|z^{(m)}(t)|_{H}^{2})ds\Big)}\bigg]\leq \frac{3C^{2}}{2k_{1}^{2}\nu}|z^{(m)}(t)|_{H}^{4}.
\end{equation}
\noindent Since $z\in C([0,T]:H)$,
\begin{equation}\nonumber
|z^{(m)}(t)|_{H} \leq \sup_{0\leq t \leq T}|z^{(m)}(t)|_{H} = |z^{(m)}(t)|_{L^{\infty}([0,T]:H)}.
\end{equation}
\noindent Thus, by integrating from $t_{0}$ to $t$;
\begin{align}
|v^{(m)}(t)|_{V}^{2} e^{\Big(-\frac{C^{2}}{2k_{1}^{2}\nu}\int_{t_{0}}^{t}(3|v^{(m)}(t)|_{H}^{2}+10|z^{(m)}(t)|_{H}^{2})ds\Big)} - |v^{(m)}(t_{0})|_{V}^{2} \leq \frac{3C^{2}}{2k_{1}^{2}\nu}|z^{(m)}|^{4}_{L^{\infty}([0,T]:H)}(t-t_{0})\nonumber
\end{align}
\begin{align}
|v^{(m)}(t)|_{V}^{2} \leq e^{\Big(\frac{C^{2}}{2k_{1}^{2}\nu}\int_{t_{0}}^{t}(3|v^{(m)}(t)|_{H}^{2}+10|z^{(m)}(t)|_{H}^{2})ds\Big)}\bigg[|v^{(m)}(t_{0})|_{V}^{2} + \frac{3C^{2}}{2k_{1}^{2}\nu}|z^{(m)}(t)|^{4}_{L^{\infty}([0,T]:H)}(t-t_{0})\bigg].\nonumber
\end{align}
\noindent In order to get an estimate uniform in $m$ we will estimate the two terms $|v^{(m)}(t_{0})|_{V}^{2}$ and $|v^{(m)}(s)|_{H}^{2}$.
\noindent Following \eqref{inner_product_1_third_term_part_1} and \eqref{ODE_difference_bound_1} we have
\begin{equation}\nonumber
\frac{1}{2}\frac{d}{dt}|(v^{(m)})(t)|_{H}^{2} + \nu|v^{(m)}(t)|^{2}_{V} \leq \frac{C}{\nu}(|v^{(m)}(t)|^{2}_{H}|z^{(m)}(t)|^{2}_{H}+|z^{(m)}(t)|^{4}_{H})+\frac{\nu}{2}|v^{(m)}(t)|^{2}_{V}.
\end{equation}
\begin{equation}\nonumber
\frac{d}{dt}|(v^{(m)})(t)|_{H}^{2}  \leq \frac{C}{\nu}|v^{(m)}(t)|^{2}_{H}|z^{(m)}|^{2}_{L^{\infty}([0,T]:H)}+\frac{C}{\nu}|z^{(m)}|^{4}_{L^{\infty}([0,T]:H)}.
\end{equation}
\noindent Similar to above calculation we get integrating from $0$ to $t$
\begin{equation}\label{|v^{(m)}(t)|^{2}_{H}_bound}
|v^{(m)}(t)|^{2}_{H} \leq e^{\Big(\frac{C}{\nu}t|z^{(m)}|^{2}_{L^{\infty}([0,T]:H)}\Big)}(|v^{(m)}(0)|^{2}_{H} + |z^{(m)}|^{2}_{L^{\infty}([0,T]:H)}).
\end{equation}
\noindent Thus,
\begin{align}
& |v^{(m)}(t)|_{V}^{2} \leq e^{\Big(\frac{C^{2}}{2k_{1}^{2}\nu}\int_{t_{0}}^{t}(3(e^{\Big(\frac{C}{\nu}t|z^{(m)}|^{2}_{L^{\infty}([0,T]:H)}\Big)}(|v^{(m)}(0)|^{2}_{H} + |z^{(m)}|^{2}_{L^{\infty}([0,T]:H)}))+10|z^{(m)}(t)|_{H}^{2})ds\Big)}\nonumber
\\& \ \ \ \times \bigg[|v^{(m)}(t_{0})|_{V}^{2} + \frac{3C^{2}}{2k_{1}^{2}\nu}|z^{(m)}(t)|^{4}_{L^{\infty}([0,T]:H)}(t-t_{0})\bigg].\nonumber
\end{align}
\noindent Now consider \eqref{ODE_difference_bound_1};
\begin{equation}\nonumber
\frac{1}{2}\frac{d}{dt}|(v^{(m)})(t)|_{H}^{2} + \nu|v^{(m)}(t)|^{2}_{V} \leq \frac{C}{\nu}(|v^{(m)}(t)|^{2}_{H}|z^{(m)}(t)|^{2}_{H}+|z^{(m)}(t)|^{4}_{H})+\frac{\nu}{2}|v^{(m)}(t)|^{2}_{V}.
\end{equation}
\noindent We have
\begin{align}
\frac{d}{dt}|(v^{(m)})(t)|_{H}^{2} + \nu|v^{(m)}(t)|^{2}_{V} & \leq \frac{C}{\nu}|v^{(m)}(t)|^{2}_{H}|z^{(m)}(t)|^{2}_{H} + \frac{C}{\nu}|z^{(m)}|^{4}_{L^{\infty}([0,T]:H)}\nonumber
\\& \leq \frac{C}{\nu}|v^{(m)}(t)|^{2}_{H}|z^{(m)}|^{2}_{L^{\infty}([0,T]:H)} + \frac{C}{\nu}|z^{(m)}(t)|^{4}_{L^{\infty}([0,T]:H)}.\nonumber
\end{align}
\noindent Integrating from $t-\tau$ to $t$
\begin{align}
& |(v^{(m)})(t)|_{H}^{2} + \nu\int_{t-\tau}^{t}|v^{(m)}(s)|^{2}_{V}ds \nonumber
\\& \leq \frac{C}{\nu}|z^{(m)}|^{2}_{L^{\infty}([0,T]:H)}(|(v^{(m)})(t)|_{H}^{2}-|(v^{(m)})(t-\tau)|_{H}^{2}) + |(v^{(m)})(t-\tau)|_{H}^{2} + \frac{C}{\nu}\tau|z^{(m)}|^{4}_{L^{\infty}([0,T]:H)}.\nonumber
\end{align}
\begin{equation}\nonumber
\nu\int_{t-\tau}^{t}|v^{(m)}(s)|^{2}_{V}ds  \leq |(v^{(m)})(t-\tau)|_{H}^{2} + \frac{C}{\nu}|z^{(m)}|^{2}_{L^{\infty}([0,T]:H)}|(v^{(m)})(t)|_{H}^{2} + \frac{C}{\nu}\tau|z^{(m)}|^{4}_{L^{\infty}([0,T]:H)}.
\end{equation}
\noindent Due to \eqref{|v^{(m)}(t)|^{2}_{H}_bound}
\begin{align}
& \int_{t-\tau}^{t}|v^{(m)}(s)|^{2}_{V}ds \nonumber
\\& \leq \frac{1}{\nu}[|(v^{(m)})(t-\tau)|_{H}^{2} + \frac{C}{\nu}|z^{(m)}|^{2}_{L^{\infty}([0,T]:H)}|(v^{(m)})(t)|_{H}^{2} + \frac{C}{\nu}\tau|z^{(m)}|^{4}_{L^{\infty}([0,T]:H)}] \nonumber
\\& \leq \frac{1}{\nu}\bigg[\bigg(e^{\Big(\frac{C}{\nu}(t-\tau)|z^{(m)}|^{2}_{L^{\infty}([0,T]:H)}\Big)}(|v^{(m)}(0)|^{2}_{H} + |z^{(m)}|^{2}_{L^{\infty}([0,T]:H)})\bigg) +  \frac{C}{\nu}\tau|z^{(m)}|^{4}_{L^{\infty}([0,T]:H)}\nonumber
\\& \ \ \ +\frac{C}{\nu}|z^{(m)}|^{2}_{L^{\infty}([0,T]:H)}\bigg( e^{\Big(\frac{C}{\nu}t|z^{(m)}|^{2}_{L^{\infty}([0,T]:H)}\Big)}(|v^{(m)}(0)|^{2}_{H} + |z^{(m)}|^{2}_{L^{\infty}([0,T]:H)})\bigg)\bigg]\nonumber.
\end{align}
\noindent By the mean value theorem there exists $t_{0}\in [t-\tau,t]$ such that $\int_{t-\tau}^{t}|v^{(m)}(s)|^{2}_{V}ds = \tau |v^{(m)}(t_{0})|^{2}_{V}$. Thus,
\begin{align}
& |v^{(m)}(t_{0})|^{2}_{V} \nonumber
\\& \leq \frac{1}{\nu\tau}\bigg[|v^{(m)}(0)|^{2}_{H}\bigg( e^{\Big(\frac{C}{\nu}(t-\tau)|z^{(m)}|^{2}_{L^{\infty}([0,T]:H)}\Big)} + \frac{C}{\nu}|z^{(m)}|^{2}_{L^{\infty}([0,T]:H)}e^{\Big(\frac{C}{\nu}t|z^{(m)}|^{2}_{L^{\infty}([0,T]:H)}\Big)} \bigg)\nonumber
\\& \ \ \ + |z^{(m)}|^{2}_{L^{\infty}([0,T]:H)} + \frac{C}{\nu}(1+\tau)|z^{(m)}|^{4}_{L^{\infty}([0,T]:H)} \bigg]\nonumber
\end{align}
\begin{align}\label{|v^{(m)}(t_{0})|^{2}_{V}_bound}
& |v^{(m)}(t_{0})|^{2}_{V} \nonumber
\\& \leq \frac{1}{\nu\tau}\bigg[|v^{(m)}(0)|^{2}_{H}\bigg( e^{\Big(\frac{C}{\nu}T|z^{(m)}|^{2}_{L^{\infty}([0,T]:H)}\Big)} + \frac{C}{\nu}|z^{(m)}|^{2}_{L^{\infty}([0,T]:H)}e^{\Big(\frac{C}{\nu}T|z^{(m)}|^{2}_{L^{\infty}([0,T]:H)}\Big)} \bigg)\nonumber
\\& \ \ \ + |z^{(m)}|^{2}_{L^{\infty}([0,T]:H)} + \frac{C}{\nu}(1+T)|z^{(m)}|^{4}_{L^{\infty}([0,T]:H)} \bigg].
\end{align}
\noindent Consider diadic intervals of the form $[\frac{T}{2^{j+1}},\frac{T}{2^{j}}]$ for $j=0,1,2,\ldots$ such that $\tau=\frac{T}{2^{j+1}}$. Apply \eqref{|v^{(m)}(t_{0})|^{2}_{V}_bound} for $t\in [\frac{T}{2^{j+1}},\frac{T}{2^{j}}]$. Then
\begin{align}
& |v^{(m)}(t)|^{2}_{V} \nonumber
\\& \leq \frac{1}{\nu\tau}\bigg[|v^{(m)}(0)|^{2}_{H}\bigg( e^{\Big(\frac{C}{\nu}T|z^{(m)}|^{2}_{L^{\infty}([0,T]:H)}\Big)} + \frac{C}{\nu}|z^{(m)}|^{2}_{L^{\infty}([0,T]:H)}e^{\Big(\frac{C}{\nu}T|z^{(m)}|^{2}_{L^{\infty}([0,T]:H)}\Big)} \bigg)\nonumber
\\& \ \ \ + |z^{(m)}|^{2}_{L^{\infty}([0,T]:H)} + \frac{C}{\nu}(1+T)|z^{(m)}|^{4}_{L^{\infty}([0,T]:H)} \bigg]\nonumber.
\end{align}
\noindent Since $t\in [\frac{T}{2^{j+1}},\frac{T}{2^{j}}]$, $\tau\leq t \leq 2\tau$. Thus,
\begin{align}
& 2\tau|v^{(m)}(t)|^{2}_{V} \nonumber
\\& \leq \frac{2}{\nu}\bigg[|v^{(m)}(0)|^{2}_{H}\bigg( e^{\Big(\frac{C}{\nu}T|z^{(m)}|^{2}_{L^{\infty}([0,T]:H)}\Big)} + \frac{C}{\nu}|z^{(m)}|^{2}_{L^{\infty}([0,T]:H)}e^{\Big(\frac{C}{\nu}T|z^{(m)}|^{2}_{L^{\infty}([0,T]:H)}\Big)} \bigg)\nonumber
\\& \ \ \ + |z^{(m)}|^{2}_{L^{\infty}([0,T]:H)} + \frac{C}{\nu}(1+T)|z^{(m)}|^{4}_{L^{\infty}([0,T]:H)} \bigg].\nonumber
\end{align}
\begin{align}\label{t|v^{(m)}(t)|^{2}_{V}_bound}
& t|v^{(m)}(t)|^{2}_{V} \nonumber
\\& \leq \frac{2}{\nu}\bigg[|v^{(m)}(0)|^{2}_{H}\bigg( e^{\Big(\frac{C}{\nu}T|z^{(m)}|^{2}_{L^{\infty}([0,T]:H)}\Big)} + \frac{C}{\nu}|z^{(m)}|^{2}_{L^{\infty}([0,T]:H)}e^{\Big(\frac{C}{\nu}T|z^{(m)}|^{2}_{L^{\infty}([0,T]:H)}\Big)} \bigg)\nonumber
\\& \ \ \ + |z^{(m)}|^{2}_{L^{\infty}([0,T]:H)} + \frac{C}{\nu}(1+T)|z^{(m)}|^{4}_{L^{\infty}([0,T]:H)} \bigg].
\end{align}
\noindent Since \eqref{t|v^{(m)}(t)|^{2}_{V}_bound} is independent of $j$ it holds for all $j=0,1,2,\ldots$. Thus \eqref{t|v^{(m)}(t)|^{2}_{V}_bound} holds for all $t\in [0,T]$. Now let $m\rightarrow \infty$. By a Galerkin approximation method,
\begin{align}\label{t|v(t)|^{2}_{V}_bound}
& t|v(t)|^{2}_{V} \nonumber
\\& \leq \frac{2}{\nu}\bigg[|v(0)|^{2}_{H}\bigg( e^{\Big(\frac{C}{\nu}T|z|^{2}_{L^{\infty}([0,T]:H)}\Big)} + \frac{C}{\nu}|z|^{2}_{L^{\infty}([0,T]:H)}e^{\Big(\frac{C}{\nu}T|z|^{2}_{L^{\infty}([0,T]:H)}\Big)} \bigg)\nonumber
\\& \ \ \ + |z|^{2}_{L^{\infty}([0,T]:H)} + \frac{C}{\nu}(1+T)|z|^{4}_{L^{\infty}([0,T]:H)} \bigg].\nonumber
\end{align}
\noindent Thus,
\begin{equation}\nonumber
t|v(t)|^{2}_{V} \leq C(\nu,|v(0)|^{2}_{H},|z|^{2}_{L^{\infty}([0,T]:H)},T)
\end{equation}
\noindent Now recalling that $$v^{(m)}(t)=u^{(m)}(t)-z^{(m)}(t)$$ we have the result for $u(t)$ and showing the result holds for the stochastic case will be similar to above step 7.
\end{proof}




\section[Some Statistical Properties of the Dyadic Model]{Some Statistical Properties of the Dyadic Model}

\noindent We start this section by introducing some preliminaries on invariant measures. We will prove the existence of invariant measures for the shell model and establish a balance relation. Using these invariant measures, we will then study on statistical properties related to the structure functions for the dyadic model. This is related to the K41 theory of Kolmogorov, (see, \cite{Kolmogorov_1941_a}, \cite{Kolmogorov_1941_b} and \cite{Frisch}). For further reading on invariant measures on infinite dimensional setting we refer to \cite{DaPrato_2006} and \cite{DaPrato_2005} and for a comprehensive reading we refer to \cite{DaPrato_1996}. Applications of invariant measures for stochastic Navier-Stokes equations we refer to \cite{Ergodicity_Flandoli} and \cite{Dissipativity_Flandoli}. We refer to \cite{Strong_feller}, \cite{BFerrario} and \cite{Benedetta} for some other interesting results.

\subsection{Preliminaries on Invariant Measures}

\noindent Let $H$ be a Hilbert space. We denote by $\mathrm{B}_{b}(H)$ the Banach space of all Borel bounded mappings $\phi :H\rightarrow \mathbb{R}$ and by $\mathrm{C}_{b}(H)$ the Banach space of all uniformly continuous and bounded mappings $\phi :H\rightarrow \mathbb{R}$. These Banach spaces are endowed with the norm
\begin{equation*}
\|\phi\|_{0} = \sup_{x\in H}|\phi(x)|.
\end{equation*}
\noindent We denote $\mathrm{L}(\mathrm{B}_{b}(H))$ as the Banach space of all linear bounded operators $P :\mathrm{B}_{b}(H) \rightarrow \mathrm{B}_{b}(H)$ and $\mathrm{L}(\mathrm{C}_{b}(H))$ as the Banach space of all linear bounded operators $P :\mathrm{C}_{b}(H) \rightarrow \mathrm{C}_{b}(H)$. We use $\mathcal{M}_{1}(H)$ to denote the space of all probability measures on $(H, \mathfrak{B}(H))$ where $\mathfrak{B}(H)$ is the $\sigma$-algebra of all Borel subsets of $H$.
\begin{definition}(\emph{Markov Semigroup})\label{Markov_semigroup}
A \emph{Markov semigroup} $P_{t}$ on $\mathrm{B}_{b}(H)$ is a mapping $[0,+\infty)\rightarrow \mathrm{L}(\mathrm{B}_{b}(H))$, $t\rightarrow P_{t}$ such that
\begin{enumerate}
\item $P_{0}=I$,
\item $P_{t+s}=P_{t}P_{s}$ for all $t,s\geq0$,
\item for any $t\geq0$ and $x\in H$ there exists a probability measure $\pi_{t}(x,\cdot)\in \mathcal{M}_{1}(H)$ such that
\begin{equation}\nonumber
(P_{t}\phi)(x) = \int_{H}\phi(y)\pi_{t}(x,dy)\ \ \text{ for all } \phi\in \mathrm{B}_{b}(H),
\end{equation}
\item for any $x\in H$ and $\phi\in \mathrm{C}_{b}(H)$ the mapping $t\rightarrow P_{t}\phi(x)$ is continuous.
\end{enumerate}
\end{definition}

\subsection{Existence of Invariant Measures}

\begin{definition}(\emph{Invariant Measure})\label{Invariant_Measure}
\noindent A probability measure $\mu \in \mathcal{M}_{1}(H)$ is said to be \emph{invariant} for a Markov semigroup $P_{t}$ if
\begin{equation}\nonumber
\int_{H}P_{t}\phi \ d\mu \quad = \quad \int_{H}\phi\ d\mu \ \ \text{ for all } \phi \in \mathrm{B}_{b}(H) \text{ and } t\geq 0.
\end{equation}
\end{definition}
\begin{definition}(\emph{Feller})\label{Feller_Property}
A Markov semigroup $P_{t}$ is called \emph{Feller} if $P_{t}\phi \ \in \  \mathrm{C}_{b}(H)$ whenever $\phi \ in \ \mathrm{C}_{b}(H)$ and for any $t\geq0$.
\end{definition}
\begin{definition}(\emph{Tightness})
A subset $\Lambda \subset \mathcal{M}_{1}(H)$ is said to be \emph{tight} if there exists an increasing sequence $(K_{n})$ of compact sets of $H$ such that
\begin{equation}\nonumber
\lim_{n\rightarrow \infty}\mu(K_{n}) \ = \ 1 \ \text{ uniformly on } \Lambda
\end{equation}
or equivalently if for any $\epsilon>0$ there exists a compact set $K_{\epsilon}$ such that
\begin{equation}\nonumber
\mu(K_{\epsilon}) \ \geq 1 \ - \ \epsilon, \quad \mu \ \in \ \Lambda.
\end{equation}
\end{definition}
\begin{theorem} (Krylov-Bogoliubov Theorem)\label{Krylov-Bogoliubov_Theorem}
For any $T>0$, let
\begin{equation}\label{set}
\mu_{T}(E) \ = \ \frac{1}{T} \int_{0}^{T}\pi_{t}(x_{0},E)dt, \ E\in \mathfrak{B}(H), x_{0}\in H.
\end{equation}
Let $P_{t}$ be a Markov Feller semigroup. Suppose for some $x_{0}\in H$ the set $(\mu_{T})_{T\geq0}$ defined by \eqref{set} is tight. Then there exists an invariant measure for $P_{t}$.
\end{theorem}
\begin{proof}
See \cite{DaPrato_1996}, page 21.
\end{proof}
\subsection{Invariant Measures for Shell Models}
\begin{theorem}
There exists an invariant measure for the shell model \eqref{shell_model_differential_form_componentwise}.
\end{theorem}
\begin{proof}
\noindent Following the Theorem \ref{existence_weak_solution} and \ref{uniqueness_of_the_solution} we have that for any initial condition $x=u(0)\in H$ there is a unique continuous adapted solution in $H$. Using the continuous dependence of the solution with respect to the initial condition in the Theorem \ref{uniqueness_of_the_solution}, we get that,
$$u^{x_{n}}(t)\rightarrow u^{x}(t) \text{ whenever } x_{n}\rightarrow x \text{ in } H, \text{ P-a.s.}$$
for all $t\geq 0$. Thus, we can define a map
$$P_{t}:\mathrm{B}_{b}(H)\rightarrow \mathrm{B}_{b}(H)$$
by the formula
$$(P_{t}\phi)(x)=\mathbb{E}[\phi(u^{x}(t))].$$
By the Lebesgue dominated convergence theorem, we deduce that the semigroup $P_{t}$ is Feller. The Markovianity follows from the Galerkin approximations. Let $\psi_{t}$ be the law of $u^{x}(t)$, $T>0$ and define the probability measure $\mu_{T}$ on $H$ as
\begin{equation}\nonumber
\mu_{T}=\frac{1}{T}\int_{0}^{T}\psi_{s}ds.
\end{equation}
\noindent By applying the Chebyshev's inequality and the Theorem \ref{estimate_for_expectation} we get,
\begin{align}
\mu_{T}(|x|^{2}_{V}\geq R)&=\frac{1}{T}\int_{0}^{T}\psi(|x|^{2}_{V}\geq R)ds\nonumber
\\& \leq \frac{1}{T}\int_{0}^{T}\frac{\mathbb{E}[|u^{x}(s)|^{2}_{V}]}{R}ds\nonumber
\\& \leq \frac{C}{R}.\nonumber
\end{align}
\noindent Hence $(\mu_{T})$ is tight in $V$. Since the semigroup $P_{t}$ is Feller, by the Krylov-Bogoliubov Theorem \ref{Krylov-Bogoliubov_Theorem} there exists an invariant measure in $H$.
\end{proof}
%
\subsection{Balance Relations for the Dyadic Model}

\noindent Given $\nu>0$, let $\mu^{\nu}$ be any invariant measure for the dyadic model and let $\mathbb{E}^{\nu}$ denotes the corresponding expectation. Let us introduce the following quantities;
\begin{equation}\label{varepsilon_n}
\varepsilon_{n} = \nu k_{n}^{2}\mathbb{E}^{\nu}[|u_{n}|^{2}],
\end{equation}
\begin{equation}\label{phi_n}
\phi_{n} = k_{n}\mathbb{E}^{\nu}[u_{n}^{2}u_{n+1}].
\end{equation}
\noindent We interpret $\varepsilon_{n}$ as the mean rate of energy dissipation at scale $k_{n}^{-1}$ and $\phi_{n}$ as the mean rate of energy flux from scale $k_{n}^{-1}$ to smaller scales. We then observe that $\phi_{n-1} = k_{n-1}\mathbb{E}^{\nu}[u_{n-1}^{2}u_{n}]$ is the mean rate of energy flux from larger scales to scale $k_{n}^{-1}$. We interpret $\sigma^{2}$ as the mean rate of energy injection at scale $k_{n}^{-1}$. We note that the quantities $\mathbb{E}^{\nu}[u_{n}^{2}]$ and $\mathbb{E}^{\nu}[u_{n}^{2}u_{n+1}]$ are finite. We have the following balance relation.
\begin{proposition}(Balance Relation)\label{balance_relation}
We have
\begin{equation}\label{balance_relation_equation_1}
\nu k_{n}^{2}\mathbb{E}^{\nu}[|u_{n}|^{2}] + k_{n}\mathbb{E}^{\nu}[u_{n}^{2}u_{n+1}] = k_{n-1}\mathbb{E}^{\nu}[u_{n-1}^{2}u_{n}] + \frac{1}{2}\sigma_{n}^{2}
\end{equation}
or using the notations introduced above
\begin{equation}\label{balance_relation_equation_2}
\varepsilon_{n} + \phi_{n} = \phi_{n-1} + \frac{1}{2}\sigma_{n}^{2} \ \text{ for all } n\geq 1
\end{equation}
for \eqref{shell_model_differential_form_componentwise}.
\end{proposition}
\begin{proof}
Suppose $u(t)$ is a stationary solution associated with the invariant measure $\mu^{\nu}$. Apply Ito formula for $|u_{n}|^{2}$. We get
\begin{align}
d|u_{n}|^{2} &= 2u_{n}du + \frac{1}{2}2\sigma_{n}^{2}dt\nonumber
\\&= 2u_{n}[-\nu k_{n}^{2}u_{n}dt + (k_{n-1}u_{n-1}^{2}-k_{n}u_{n}u_{n+1})dt + \sigma_{n}d\beta_{n}] + \sigma_{n}^{2}dt\nonumber
\\&= -2\nu k_{n}^{2}u_{n}^{2}dt + 2k_{n-1}u_{n-1}^{2}u_{n}dt - 2k_{n}u_{n}^{2}u_{n+1}dt + 2\sigma_{n}u_{n}d\beta_{n} + \sigma_{n}^{2}dt.\nonumber
\end{align}
\noindent Integrating from $0$ to $t$ we obtain
\begin{equation}\nonumber
|u_{n}(t)|^{2} = |u_{n}(0)|^{2} - 2\int_{0}^{t}\nu k_{n}^{2}u_{n}^{2}ds + \int_{0}^{t}k_{n-1}u_{n-1}^{2}u_{n}ds -2\int_{0}^{t}k_{n}u_{n}^{2}u_{n+1}ds + 2\int_{0}^{t}\sigma_{n}u_{n}d\beta_{n} + \sigma_{n}^{2}t.
\end{equation}
\noindent Taking expectation with respect to the invariant measure $\mu^{\nu}$ we obtain
\begin{align}
\mathbb{E}^{\nu}[|u_{n}|^{2}] &= \mathbb{E}^{\nu}[|u_{n}(0)|^{2}] - 2\int_{0}^{t}\nu k_{n}^{2}\mathbb{E}^{\nu}[|u_{n}|^{2}]ds\nonumber
\\ & \ + \int_{0}^{t}k_{n-1}\mathbb{E}^{\nu}[u_{n-1}^{2}u_{n}]ds -2\int_{0}^{t}k_{n}\mathbb{E}^{\nu}[u_{n}^{2}u_{n+1}]ds + \sigma_{n}^{2}t.\nonumber
\end{align}
\noindent Since $\mathbb{E}^{\nu}[|u_{n}(t)|^{2}] = \mathbb{E}^{\nu}[|u_{n}(0)|^{2}]$ by stationarity of $u$, we have
\begin{equation}\nonumber
\int_{0}^{t}\nu k_{n}^{2}\mathbb{E}^{\nu}[|u_{n}|^{2}]ds  + \int_{0}^{t}k_{n}\mathbb{E}^{\nu}[u_{n}^{2}u_{n+1}]ds = \int_{0}^{t}k_{n-1}\mathbb{E}^{\nu}[u_{n-1}^{2}u_{n}]ds + \frac{1}{2}\sigma_{n}^{2}t.
\end{equation}
\noindent By stationarity of $u$, the integrands are independent of $s$. Thus,
\begin{equation}\nonumber
\nu k_{n}^{2}\mathbb{E}^{\nu}[|u_{n}|^{2}] + k_{n}\mathbb{E}^{\nu}[u_{n}^{2}u_{n+1}] = k_{n-1}\mathbb{E}^{\nu}[u_{n-1}^{2}u_{n}] + \frac{1}{2}\sigma_{n}^{2}.
\end{equation}
\end{proof}
\subsection[Some Statistical Properties]{Some Statistical Properties}

\noindent We start this section with a definition and for more details we refer to \cite{Ditlevsena}, \cite{Gallavotti}, \cite{Biferale_2003}, \cite{Benzi}, \cite{Scaling_Flandoli} and \cite{GOY_Flandoli}.

\begin{definition}(\emph{$p$-Order Structure Function}) For any $\nu>0$ let $\mu^{\nu}$ be any invariant measure of the dyadic model and $\mathbb{E}^{\nu}$ denotes the corresponding expectation. The expression
\begin{equation}\nonumber
S_{p}^{\nu}(n):=\mathbb{E}^{\nu}[|u_{n}^{\nu}|^{p}]
\end{equation}
is called the \emph{$p$-order structure function}.
\end{definition}
\noindent One is interested in the scaling behavior of the form
$$\mathbb{E}^{\nu}[|u_{n}^{\nu}|^{p}] \sim k_{n}^{-\zeta_{p}}.$$
This can be true only in an intermediate range of $n'$s according to Kolmogorov conjecture, we refer to \cite{GOY_Flandoli} and references therein for more details. We work in a range of the form $n\in [n_{-}(\nu),n_{+}(\nu)]$ with $n_{-},n_{+}:(0,1)\rightarrow \mathbb{N}$ such that
\begin{enumerate}
	\item $n_{-}(\nu)<n_{+}(\nu),$
	\item $\displaystyle\lim_{\nu\rightarrow 0}\frac{n_{-}(\nu)}{n_{+}(\nu)}=0.$
\end{enumerate}
\noindent Following \cite{GOY_Flandoli}, on $n_{+}(\nu)$ we impose the condition
$$ \lim_{\nu\rightarrow 0} \frac{n_{+}(\nu)}{\log_{2}(\nu)}=\frac{3}{4}.$$
\noindent We now define the asymptotic exponents.
\begin{definition}(\emph{Asymptotic Exponents of Order $p$})
Let $(\mu^{\nu})_{\nu>0}$ be a set of invariant measures for the dyadic model. Let $n_{-},n_{+}:(0,1)\rightarrow \mathbb{N}$ such that $n_{-}(\nu)<n_{+}(\nu),$ and $\displaystyle\lim_{\nu\rightarrow 0}\frac{n_{-}(\nu)}{n_{+}(\nu)}=0.$ Let $$R=\{(\nu,n)\in (0,1)\times \mathbb{N}: n\in [n_{-}(\nu),n_{+}(\nu)]\}.$$ For $p\geq 0$, for the given set of invariant measures $\mu^{\nu}$, and for $(\nu,n)\in R$, we define the following quantities.
$$\zeta_{p}^{+}:=-\liminf_{\stackrel{\nu \rightarrow 0}{(\nu,n)\in R}}\frac{1}{n}\log_{2}\mathbb{E}^{\nu}[|u_{n}|^{p}],$$
$$\zeta_{p}^{-}:=-\limsup_{\stackrel{\nu \rightarrow 0}{(\nu,n)\in R}}\frac{1}{n}\log_{2}\mathbb{E}^{\nu}[|u_{n}|^{p}].$$
When $\zeta_{p}^{+}=\zeta_{p}^{-}$ we define the common number $\zeta_{p}$ the \emph{asymptotic exponent of order $p$}.
\end{definition}
\noindent The value of $\zeta_{p}$ is not known analytically but based on numerical results and physical intuition it is agreed that $\zeta_{3}=1$ and $\zeta_{p}<\frac{p}{3}$ for large $p$. However, the value of $\zeta_{2}$ is not clear. Kolmogorov theory for 3D turbulence claimed that $\zeta_{2}=\frac{2}{3}$, but numerical simulations suggest that $\zeta_{2}>\frac{2}{3}$. Thus it is interesting to have an analytical proof for the value of $\zeta_{2}$. In \cite{GOY_Flandoli}, some necessary and sufficient conditions have been studied for both $\zeta_{2}=\frac{2}{3}$ and $\zeta_{2}\geq \frac{2}{3}$ for the GOY model. However, an analytical proof is lacking for either case. Further, these results were obtained under an assumption which is not derived from the balance law. Our aim is to improve these results in the same line for the dyadic model. However, we obtain some rigorous results about the statistical relationships of the variables $u_{n}$ and $u_{n+1}$ for the dyadic model.
%
%
%
%
%
\subsection{Some Results on K41 Theory}
\begin{lemma}\label{phi_n_md}
Consider the balance relation given in \eqref{balance_relation_equation_2}. Assume $\sigma = \sigma_{1} \neq 0$ and $\sigma_{n} = 0$ for all $n\geq 2$ in \eqref{shell_model_differential_form_componentwise}. Then, $$\phi_{n+1} \leq \phi_{n}$$ for all $n\geq 1.$
\end{lemma}
\begin{proof}
From the balance relation, we have $\varepsilon_{n} + \phi_{n} = \phi_{n-1}$ for all $n\geq2$. Since $\varepsilon_{n} \geq 0$ for all $n\geq 2,$ we have $\phi_{n} \leq \phi_{n-1}$ for all $n\geq2$. Thus, $\phi_{n+1} \leq \phi_{n}$ for all $n\geq1$.
\end{proof}
\begin{theorem}\label{upper_bound_for_phi_n}
Consider the balance relation given in \ref{balance_relation_equation_2}. Assume $\sigma = \sigma_{1} \neq 0$ and $\sigma_{n} = 0$ for all $n\geq 2$ in \eqref{shell_model_differential_form_componentwise}. Then, $$\phi_{n} \leq \frac{1}{2}\sigma^{2}$$ for all $n\geq 1.$
\end{theorem}
\begin{proof}
We have $\varepsilon_{n} \geq 0$ for all $n\geq 1$. Thus, $\phi_{n} \leq \phi_{n-1} + \frac{1}{2}\sigma_{n}^{2}$. For $n=1,$ we have $\phi_{1} \leq \frac{1}{2}\sigma^{2}$ as $\phi_{0}=0$. For $n=2$, we have $\phi_{2} \leq \phi_{1}$ as $\sigma_{2}=0$ and thus $\phi_{2} \leq \frac{1}{2}\sigma^{2}$. Assume for $n=p$, $\phi_{p}\leq \frac{1}{2}\sigma^{2},$ then $\phi_{p+1} \leq \frac{1}{2}\sigma^{2}$ since $\phi_{P+1} \leq \phi_{p}$ by Lemma \ref{phi_n_md}. Thus, $\phi_{n} \leq \frac{1}{2}\sigma^{2}$ for all $n\geq 1$.
\end{proof}
\begin{remark}
Consider the balance relation given in \ref{balance_relation}. Assume $\sigma = \sigma_{1} \neq 0$ and $\sigma_{n} = 0$ for all $n\geq 2$ for \eqref{shell_model_differential_form_componentwise}. Since $\phi_{n} = k_{n}\mathbb{E}^{\nu}[u_{n}^{2}u_{n+1}]$ and $\phi_{n} \leq \frac{1}{2}\sigma^{2},$ we have $k_{n}\mathbb{E}^{\nu}[u_{n}^{2}u_{n+1}] < \frac{1}{2}\sigma^{2}$. Thus, $$\mathbb{E}^{\nu}[u_{n}^{2}u_{n+1}] \leq k_{n+1}^{-1}\sigma^{2}.$$
\end{remark}
\noindent In order to prove an estimate for $|\phi_{n}|$ we will prove the following Lemmas. We are unable to get an estimate valid for all $n$ and we need to impose some restriction on the range of $n$s. The following Lemma indicates that the mean dissipation rate is exponentially smaller than the mean flux rate for a certain range of $n$s.
\begin{lemma}\label{upper_bound_for_varepsilon}
Assume that there exist $\overline{\nu} >0$ and $\gamma>0$ such that $$\frac{\mathbb{E}^{\nu}[|u_{n}|^{2}]}{|\mathbb{E}^{\nu}[u_{n}^{2}u_{n+1}]|^{2/3}} \leq \gamma$$  \noindent for all $(\nu,n) \in  (0,\overline{\nu}]\times[n_{-}(\nu),n_{+}(\nu)]$ and assume that $$\displaystyle\liminf_{\nu \rightarrow 0}\frac{n_{+}(\nu)}{\log_{2}\nu} > -\frac{3}{4}.$$\\
\noindent Then there exists $\alpha \in (0,1)$ depending on $n_{+}$ for all $\delta >0$, there exists $\nu_{0}>0$ depending on $\delta, \gamma, n_{+}$ such that $$0\leq\varepsilon_{n} \leq \delta \alpha^{n} |\phi_{n}|^{2/3}$$ for all $(\nu, n) \in (0,\nu_{0}]\times [n_{-}(\nu), n_{+}(\nu)].$
\end{lemma}
\begin{proof}
Recall that $\phi_{n} = k_{n}\mathbb{E}^{\nu}[u_{n}^{2}u_{n+1}].$ Thus, $|\mathbb{E}^{\nu}[u_{n}^{2}u_{n+1}]|^{2/3} = k_{n}^{-2/3}|\phi_{n}|^{2/3}$. Using the assumption $\displaystyle\frac{\mathbb{E}^{\nu}[|u_{n}|^{2}]}{|\mathbb{E}^{\nu}[u_{n}^{2}u_{n+1}]|^{2/3}} \leq \gamma,$ we get  $\mathbb{E}^{\nu}[|u_{n}|^{2}] \leq \gamma |\mathbb{E}^{\nu}[u_{n}^{2}u_{n+1}]|^{2/3} = \gamma k_{n}^{-2/3} |\phi_{n}|^{2/3}$. Thus we have,
\begin{align}
\varepsilon_{n} = \nu k_{n}^{2}\mathbb{E}^{\nu}[|u_{n}|^{2}] &\leq \gamma \nu k_{n}^{2}k_{n}^{-2/3}|\phi_{n}|^{2/3} \nonumber
\\ &= \gamma \nu k_{n}^{4/3}|\phi_{n}|^{2/3} \nonumber
\\ &= (\gamma \nu^{\eta_{1}})(\nu^{1-\eta_{1}}k_{n}^{4/3+\eta_{2}})k_{n}^{-\eta_{2}}|\phi_{n}|^{2/3} \nonumber
\\ &\leq (\gamma \nu^{\eta_{1}})(\nu^{1-\eta_{1}}k_{n_{+}(\nu)}^{4/3+\eta_{2}})k_{n}^{-\eta_{2}}|\phi_{n}|^{2/3}\nonumber
\end{align}
\noindent for all $n\in [n_{-}(\nu),n_{+}(\nu)],$ $ \nu \in (0,\overline{\nu}]$ and for all $\eta_{1}, \eta_{2} >0.$\\
\noindent Now consider the assumption $\displaystyle\liminf_{\nu \rightarrow 0}\frac{n_{+}(\nu)}{\log_{2}\nu} > -\frac{3}{4}.$ Thus, there exists $\epsilon >0,$ and $\nu'>0$ such that $\frac{n_{+}(\nu)}{\log_{2}\nu} \geq -\frac{3}{4} + \epsilon$ for all $\nu<\nu'$. Thus, $n_{+}(\nu) \leq (-\frac{3}{4} + \epsilon)\log_{2}(\nu)$, $(0<\nu<1)$. Thus, $2^{n_{+}(\nu)} \leq \nu^{(-3/4 + \epsilon)}$ and $2^{n_{+}(\nu)}\nu^{(3/4 -\epsilon)} \leq 1$. Assuming $k_{0}=1$ we have $k_{n_{+}(\nu)}\nu^{(3/4 -\epsilon)} \leq 1$. Now choose $\eta_{1}, \eta_{2}>0$ such that $\frac{1-\eta_{1}}{4/3+\eta_{2}}=\frac{3}{4}-\epsilon$. Then,
\begin{center}
$(\nu^{1-\eta_{1}}k_{n_{+}(\nu)}^{4/3+\eta_{2}}) = (\nu^{\frac{1-\eta_{1}}{4/3+\eta_{2}}}k_{n_{+}(\nu)})^{4/3+\eta_{2}} = (\nu^{3/4-\epsilon}k_{n_{+}(\nu)})^{4/3+\eta_{2}} \leq 1.$
\end{center}
\noindent Then (assuming $k_{0}=1$) we have for all $\nu\leq\nu_{0}=\min\{\overline{\nu},\nu'\}$
\begin{equation}\nonumber
\epsilon_{n} \leq (\gamma \nu^{\eta_{1}})(k_{n}^{-\eta_{2}})|\phi_{n}|^{2/3} = (\gamma \nu^{\eta_{1}})\bigg(\frac{1}{2^{\eta_{2}}}\bigg)^{n}|\phi_{n}|^{2/3}.
\end{equation}
By taking $\delta = (\gamma \nu^{\eta_{1}})$ and $\alpha = \big(\frac{1}{2^{\eta_{2}}}\big) \in (0,1)$ we get $\varepsilon_{n} \leq \delta \alpha^{n} |\phi_{n}|^{2/3}$.
\end{proof}
\begin{remark}
Note that $\alpha$ depends on $n_{+}(\nu)$ and we can make $\delta>0$ arbitrary small by choosing $\nu_{0}$ small.
\end{remark}
\noindent For the next Lemma, we need the same assumptions we used in the above Lemma \ref{upper_bound_for_varepsilon} as we will use the estimate for $\varepsilon_{1}$. Further, we will follow an iterating procedure on $n$ so we set $n_{-}(\nu)=1$.
\begin{lemma}\label{lower_bound_for_phi_1}
Suppose $n_{-}(\nu)=1$, there exist $\overline{\nu} >0$, $\gamma>0$ such that $\displaystyle\frac{\mathbb{E}^{\nu}[|u_{n}|^{2}]}{|\mathbb{E}^{\nu}[u_{n}^{2}u_{n+1}]|^{2/3}} \leq \gamma$ for all $(\nu,n) \in  (0,\overline{\nu}]\times[1,n_{+}(\nu)]$, and $\displaystyle\liminf_{\nu \rightarrow 0}\frac{n_{+}(\nu)}{\log_{2}\nu} > -\frac{3}{4}.$ Consider the balance relation given in \ref{balance_relation}. If $\delta < \frac{3\sigma^{2/3}}{2\alpha}$ then $$\phi_{1}\geq 0.$$
\end{lemma}
\begin{proof}
\noindent From the balance relation \eqref{balance_relation_equation_2}, we have $\varepsilon_{1} + \phi_{1} = \frac{1}{2}\sigma^{2}.$ From \ref{upper_bound_for_varepsilon} we have the estimate $\varepsilon_{1} \leq \delta \alpha |\phi_{1}|^{2/3}$. Thus, $\phi_{1} + \delta \alpha |\phi_{1}|^{2/3} \geq \frac{1}{2}\sigma^{2}$.\\
\noindent Now assume $\phi_{1}<0.$ Consider the function $f(x) = x + \delta \alpha |x|^{2/3}$ for  $x<0$. Solving $f'(x)=0$ we get the only critical point $x_{0}$ at $-\frac{8}{27}\delta^{3}\alpha^{3}<0$ and we have $f''(x_{0})=-\frac{9}{8}\frac{1}{\delta^{3}\alpha^{3}}<0$. Thus at $x_{0}$, $f$ has a global maximum and the maximum value $f(x_{0}) = \frac{4}{27}\delta^{3}\alpha^{3}$. If $\delta < \frac{3\sigma^{2/3}}{2\alpha}$ then the maximum value of $f$ is less than $\frac{1}{2}\sigma^{2}$. This is a contradiction. Thus, $\phi_{1}\geq 0.$
\end{proof}
\noindent We now give a lower bound for $\phi_{n}$ using the lower bound proven in the above Lemma \ref{lower_bound_for_phi_1} for $\phi_{1}$ and employing an inductive procedure. Since we use the estimates proved in Theorem \ref{upper_bound_for_phi_n} and Lemma \ref{upper_bound_for_varepsilon} and \ref{lower_bound_for_phi_1} this estimate is valid under all the assumptions used previously.
\begin{lemma}\label{lower_bound_for_phi_n}
Consider the balance relation given in \eqref{balance_relation_equation_2}. Assume $\sigma = \sigma_{1} \neq 0$ and $\sigma_{n} = 0$ for all $n\geq 2$ in \eqref{shell_model_differential_form_componentwise}. Assume that there exist $\overline{\nu} >0$ and $\gamma>0$ such that $\displaystyle\frac{\mathbb{E}^{\nu}[|u_{n}|^{2}]}{|\mathbb{E}^{\nu}[u_{n}^{2}u_{n+1}]|^{2/3}} \leq \gamma$ for all $(\nu,n) \in  (0,\overline{\nu}]\times[1,n_{+}(\nu)]$ and assume that $\displaystyle\liminf_{\nu \rightarrow 0}\frac{n_{+}(\nu)}{\log_{2}\nu} > -\frac{3}{4}.$  If $$\delta < \frac{(1-\alpha)\sigma^{2/3}}{2^{1/3}\alpha}$$ then $$\phi_{n} \geq \frac{1}{2}\sigma^{2} - \delta\bigg(\frac{1}{2}\bigg)^{2/3}\sigma^{4/3}\alpha(1+\alpha+...+\alpha^{n-1})>0$$ for all $(\nu,n) \in  (0,\overline{\nu}]\times[1,n_{+}(\nu)]$.
\end{lemma}
\begin{proof}
Since $\phi_{1}\geq 0$ by Lemma \ref{lower_bound_for_phi_1} and $\phi_{1} \leq \frac{1}{2}\sigma^{2}$ by Theorem \ref{upper_bound_for_phi_n}, $|\phi_{1}|^{2/3} \leq (\frac{1}{2})^{2/3}\sigma^{4/3}$. From Lemma \ref{upper_bound_for_varepsilon} we have $\varepsilon_{1} \leq \delta \alpha |\phi_{1}|^{2/3}$. Thus, $\varepsilon_{1} \leq \delta \alpha (\frac{1}{2})^{2/3}\sigma^{4/3}$. From the balance relation  \ref{balance_relation} we have $\varepsilon_{1} + \phi_{1}=\frac{1}{2}\sigma^{2}.$ Thus,
$$\phi_{1} \geq \frac{1}{2}\sigma^{2}  - \delta\Big(\frac{1}{2}\Big)^{2/3}\sigma^{4/3}\alpha>0.$$
Note that again from the balance relation  \ref{balance_relation} we have $\varepsilon_{2} + \phi_{2}=\phi_{1}.$ Since $\varepsilon_{2} \leq \delta\alpha^{2}|\phi_{2}|^{2/3}$, and $\phi_{1} \geq \frac{1}{2}\sigma^{2}  - \delta(\frac{1}{2})^{2/3}\sigma^{4/3}\alpha$ following the similar reasoning we get $$\phi_{2} \geq \frac{1}{2}\sigma^{2}  - \delta\Big(\frac{1}{2}\Big)^{2/3}\sigma^{4/3}\alpha(1+\alpha)>0.$$
Now assume that $$\phi_{p} \geq \frac{1}{2}\sigma^{2}  - \delta\Big(\frac{1}{2}\Big)^{2/3}\sigma^{4/3}\alpha(1+\alpha+...+\alpha^{p-1}).$$ From the balance relation  \ref{balance_relation} we have $\phi_{p+1} + \varepsilon_{p+1} = \phi_{p}$ and using the similar argument we get$$\phi_{p+1} \geq \frac{1}{2}\sigma^{2} - \delta\Big(\frac{1}{2}\Big)^{2/3}\sigma^{4/3}\alpha(1+\alpha+...+\alpha^{p})>0.$$ \\
Thus, by the mathematical induction, the result holds for all $n\in (1,n_{+})$.
\end{proof}
\begin{remark}
Recall that $\alpha\in(0,1)$ and note that $$\frac{1}{1-\alpha}\delta\Big(\frac{1}{2}\Big)^{2/3}\sigma^{4/3}\alpha < \frac{1}{2}\sigma^{2} \text{ if and only if }\delta < \frac{(1-\alpha)\sigma^{2/3}}{2^{1/3}\alpha}.$$
\end{remark}
\begin{theorem}\label{abs_phi_n}
Consider the balance relation given in \eqref{balance_relation_equation_2}. Assume $\sigma = \sigma_{1} \neq 0$ and $\sigma_{n} = 0$ for all $n\geq 2$ in \eqref{shell_model_differential_form_componentwise}. Assume that there exist $\overline{\nu} >0$ and $\gamma>0$ such that $\displaystyle\frac{\mathbb{E}^{\nu}[|u_{n}|^{2}]}{|\mathbb{E}^{\nu}[u_{n}^{2}u_{n+1}]|^{2/3}} \leq \gamma$ for all $(\nu,n) \in  (0,\overline{\nu}]\times[1,n_{+}(\nu)]$ and assume that $\displaystyle\liminf_{\nu \rightarrow 0}\frac{n_{+}(\nu)}{\log_{2}\nu} > -\frac{3}{4}.$
Then there exists $\nu_{0}>0$ depending on $\sigma$, $\gamma$ and $n_{+}$ such that $$|\phi_{n}| \leq \frac{1}{2}\sigma^{2}$$ for all $(\nu,n) \in  (0,\nu_{0}]\times[1,n_{+}(\nu)]$.
\end{theorem}
\begin{proof}
Proof is straightforward from Lemma \ref{lower_bound_for_phi_1} and \ref{lower_bound_for_phi_n} together with Theorem \ref{upper_bound_for_phi_n}.
\end{proof}
\begin{remark}\label{phi_n_two_side_inequality}
In particular, under the assumptions of above Theorem \ref{abs_phi_n} for some constant $C>0$ we have $$C\leq \phi_{n} \leq \frac{1}{2}\sigma^{2}$$ or $$Ck_{n}^{-1}\leq \mathbb{E}^{\nu}[u_{n}^{2}u_{n+1}] \leq k_{n}^{-1}\frac{1}{2}\sigma^{2}$$ for all $(\nu,n) \in  (0,\nu_{0}]\times[1,n_{+}(\nu)].$
\end{remark}
\begin{coro}\label{corollary_1}
Under the assumptions of the Theorem \ref{abs_phi_n} we have $$\mathbb{E}^{\nu}[|u_{n}|^{2}]\leq \gamma 2^{2/3}\sigma^{4/3}k_{n}^{-2/3}$$
for all $(\nu,n) \in  (0,\nu_{0}]\times[1,n_{+}(\nu)]$.
\end{coro}
\begin{proof}
Note that the inequality $\mathbb{E}^{\nu}[u_{n}^{2}u_{n+1}] \leq k_{n}^{-1}\frac{1}{2}\sigma^{2}$ is true for all $n\in [1,n_{+}]$ in the reamrk \ref{phi_n_two_side_inequality}. Proof is immediate from this fact and the Theorem \ref{abs_phi_n}.
\end{proof}
\begin{theorem}\label{xi_2_geq_2/3_first}
Consider the balance relation given in \eqref{balance_relation_equation_2}. Assume $\sigma = \sigma_{1} \neq 0$ and $\sigma_{n} = 0$ for all $n\geq 2$ in \eqref{shell_model_differential_form_componentwise}. Further assume that there exist $\overline{\nu} >0$, $\gamma>0$ such that $\displaystyle\frac{\mathbb{E}^{\nu}[|u_{n}|^{2}]}{|\mathbb{E}^{\nu}[u_{n}^{2}u_{n+1}]|^{2/3}} \leq \gamma$  for all $(\nu,n) \in  (0,\overline{\nu}]\times[1,n_{+}(\nu)]$ and $\displaystyle\liminf_{\nu \rightarrow 0}\frac{n_{+}(\nu)}{\log_{2}\nu} > -\frac{3}{4}$. If $\gamma 2^{2/3}\sigma^{4/3}\leq 1$ then, $$\zeta^{-}_{2}\geq \frac{2}{3}$$ for all $(\nu,n) \in  (0,\overline{\nu}]\times[n_{-}(\nu),n_{+}(\nu)]$ such that $\displaystyle\limsup_{\nu \rightarrow 0}\frac{\log_{2}\nu^{-1}}{n_{-}(\nu)} <M<\infty$ for some $M>0$.
\end{theorem}
\begin{proof}
Let $\epsilon>0.$ Using the above corollary \ref{corollary_1} we have $\mathbb{E}^{\nu}[|u_{n}|^{2}]\leq \gamma 2^{2/3}\sigma^{4/3}k_{n}^{-2/3}$. Thus, $$\mathbb{E}^{\nu}[|u_{n}|^{2}]\leq k_{n}^{-2/3} \leq \nu^{-\frac{\epsilon}{2M}}k_{n}^{-\frac{2}{3}+\frac{\epsilon}{2}}.$$ Considering logarithms and then dividing both sides by $n$ we get $$\frac{1}{n}\log_{2}\mathbb{E}^{\nu}[|u_{n}|^{2}]\leq -\frac{2}{3} + \frac{\epsilon}{2}+\frac{\epsilon}{2M}\frac{\log_{2}\nu^{-1}}{n}\leq -\frac{2}{3} + \frac{\epsilon}{2}+\frac{\epsilon}{2M}\frac{\log_{2}\nu^{-1}}{n_{-}(\nu)}.$$ Then considering limit we get $$\displaystyle\limsup_{\stackrel{\nu \rightarrow 0}{(\nu,n)\in R}}\frac{1}{n}\log_{2}\mathbb{E}^{\nu}[|u_{n}|^{2}]\leq -\frac{2}{3} + \frac{\epsilon}{2}+\frac{\epsilon}{2M}\displaystyle\limsup_{\stackrel{\nu \rightarrow 0}{(\nu,n)\in R}}\frac{\log_{2}\nu^{-1}}{n_{-}(\nu)}\leq -\frac{2}{3} +\epsilon.$$ Thus, $\zeta^{-}_{2}\geq \frac{2}{3}.$
\end{proof}
\noindent We will next give sufficient and necessary conditions for $\zeta^{-}_{2}\geq \frac{2}{3}$ or $\zeta^{-}_{2}= \frac{2}{3}$. We use a slightly different assumption on the ratio $\displaystyle\frac{\mathbb{E}^{\nu}[|u_{n}|^{2}]}{|\mathbb{E}^{\nu}[u_{n}^{2}u_{n+1}]|^{2/3}}.$ We will define another quantity in order to prove the next theorem.
\begin{definition}(\emph{Flux Asymptotic Exponent})
We say the quantity defined as
$$\zeta_{3}^{flux}:=-\limsup_{\stackrel{\nu \rightarrow 0}{(\nu,n)\in R}}\frac{1}{n}\log_{2}|\mathbb{E}^{\nu}[u_{n}^{2}u_{n+1}]|$$
the \emph{flux asymptotic exponent}, when it exists.
\end{definition}
\begin{theorem}
Assume $(\nu,n) \in  (0,\overline{\nu}]\times[n_{-}(\nu),n_{+}(\nu)]$ be such that $$\nu^{-\epsilon}\leq k_{n_{+}(\nu)}\leq k_{n_{+}(\nu)}\leq \nu^{-\alpha}$$ for some $\alpha>\epsilon>0$ and $\overline{\nu}>0$ and $$\zeta_{3}^{flux}=1.$$ Then
$$\zeta^{-}_{2}\geq \frac{2}{3} \text{ if and only if } \limsup_{\stackrel{\nu \rightarrow 0}{(\nu,n)\in R}}\frac{1}{\log\nu^{-1}}\log\frac{\mathbb{E}^{\nu}[u_{n}^{2}]}{|\mathbb{E}^{\nu}[u_{n}^{2}u_{n+1}]|^{2/3}}\leq 0$$
and
$$\zeta^{-}_{2}= \frac{2}{3} \text{ if and only if } \limsup_{\stackrel{\nu \rightarrow 0}{(\nu,n)\in R}}\frac{1}{\log\nu^{-1}}\log\frac{\mathbb{E}^{\nu}[u_{n}^{2}]}{|\mathbb{E}^{\nu}[u_{n}^{2}u_{n+1}]|^{2/3}}= 0.$$
\end{theorem}
\begin{proof}
Consider the first part. Assume that $\zeta^{-}_{2}\geq \frac{2}{3}.$ I.e., $$-\limsup_{\stackrel{\nu \rightarrow 0}{(\nu,n)\in R}}\frac{1}{n}\log_{2}\mathbb{E}^{\nu}[|u_{n}|^{2}]\geq \frac{2}{3}.$$ Thus, for any $\epsilon>0$ there exits $\nu_{1}>0$ such that for all $(\nu,n) \in  (0,\nu_{1}]\times[n_{-}(\nu),n_{+}(\nu)]$, $$\frac{1}{n}\log_{2}\mathbb{E}^{\nu}[|u_{n}|^{2}]\leq -\frac{2}{3}+\frac{\epsilon}{3\alpha}.$$ Assuming $k_{0}=1$ we get
\begin{equation}\label{leq1}
\mathbb{E}^{\nu}[|u_{n}|^{2}]\leq k_{n}^{-\frac{2}{3}+\frac{\epsilon}{3\alpha}}.
\end{equation}
Since $\zeta_{3}^{flux}=1$, I.e., $$-\limsup_{\stackrel{\nu \rightarrow 0}{(\nu,n)\in R}}\frac{1}{n}\log_{2}|\mathbb{E}^{\nu}[u_{n}^{2}u_{n+1}]|=1.$$ Thus, for any $\epsilon>0$ there exits $\nu_{2}>0$ such that for all $(\nu,n) \in  (0,\nu_{2}]\times[n_{-}(\nu),n_{+}(\nu)]$, $$\frac{1}{n}\log_{2}|\mathbb{E}^{\nu}[u_{n}^{2}u_{n+1}]|\geq -1+\frac{\epsilon}{\alpha}.$$ Assuming Assuming $k_{0}=1$ we have $|\mathbb{E}^{\nu}[u_{n}^{2}u_{n+1}]|\geq k_{n}^{-1-\epsilon}.$ Thus we get
\begin{equation}\label{geq1}
|\mathbb{E}^{\nu}[u_{n}^{2}u_{n+1}]|^{2/3}\geq k_{n}^{-\frac{2}{3}+\frac{-2\epsilon}{3\alpha}}.
\end{equation}
From \eqref{leq1} and \eqref{geq1} we get, $$\frac{\mathbb{E}^{\nu}[|u_{n}|^{2}]}{|\mathbb{E}^{\nu}[u_{n}^{2}u_{n+1}]|^{2/3}} \leq k_{n}^{\epsilon/\alpha} \leq k_{n_{+}(\nu)}^{\epsilon/\alpha}.$$ Since $k_{n_{+}(\nu)}\leq \nu^{-\alpha}$ we have $k_{n_{+}(\nu)}^{\epsilon/\alpha}\leq \nu^{\epsilon}.$ Thus, $$\frac{\mathbb{E}^{\nu}[|u_{n}|^{2}]}{|\mathbb{E}^{\nu}[u_{n}^{2}u_{n+1}]|^{2/3}} \leq \nu^{\epsilon}.$$ Considering logarithms and then dividing by $\log_{2}\nu^{-}$ we get $$\frac{1}{\log\nu^{-1}}\log\frac{\mathbb{E}^{\nu}[u_{n}^{2}]}{|\mathbb{E}^{\nu}[u_{n}^{2}u_{n+1}]|^{2/3}} \leq \epsilon$$ for all $(\nu,n) \in  (0,\nu_{0}]\times[n_{-}(\nu),n_{+}(\nu)]$ where $nu_{0}=\min\{\nu_{1},\nu_{2}\}.$ Taking limit we get $$\limsup_{\stackrel{\nu \rightarrow 0}{(\nu,n)\in R}}\frac{1}{\log\nu^{-1}}\log\frac{\mathbb{E}^{\nu}[u_{n}^{2}]}{|\mathbb{E}^{\nu}[u_{n}^{2}u_{n+1}]|^{2/3}}\leq 0.$$
\bigskip
In order to prove the converse direction assume
$$\zeta^{-}_{2}\geq \frac{2}{3} \text{ if and only if } \limsup_{\stackrel{\nu \rightarrow 0}{(\nu,n)\in R}}\frac{1}{\log\nu^{-1}}\log\frac{\mathbb{E}^{\nu}[u_{n}^{2}]}{|\mathbb{E}^{\nu}[u_{n}^{2}u_{n+1}]|^{2/3}}\leq 0.$$
Thus, for all $(\nu,n) \in  (0,\nu_{1}]\times[n_{-}(\nu),n_{+}(\nu)]$, $$\frac{1}{\log\nu^{-1}}\log\frac{\mathbb{E}^{\nu}[u_{n}^{2}]}{|\mathbb{E}^{\nu}[u_{n}^{2}u_{n+1}]|^{2/3}}\leq \frac{\epsilon}{3}.$$ Thus, we have
\begin{equation}\label{leq2}
\mathbb{E}^{\nu}[u_{n}^{2}]\leq \nu^{-\frac{\epsilon}{3}} |\mathbb{E}^{\nu}[u_{n}^{2}u_{n+1}]|^{2/3}.
\end{equation}
Since $\zeta_{3}^{flux}=1$, following the same argument as above we can get
\begin{equation}\label{leq3}
|\mathbb{E}^{\nu}[u_{n}^{2}u_{n+1}]|^{2/3}\leq k_{n}^{-\frac{2}{3}+\frac{2\epsilon}{3}}.
\end{equation}
From \eqref{leq2}, \eqref{leq3} and using the assumption on $n_{-}(\nu)$ we get $$\mathbb{E}^{\nu}[u_{n}^{2}]\leq \nu^{-\frac{\epsilon}{3}} k_{n}^{-\frac{2}{3}+\frac{2\epsilon}{3}} = k_{n}^{-\frac{2}{3}+\epsilon} .$$ Thus, $$\frac{1}{n}\log_{2}\mathbb{E}^{\nu}[u_{n}^{2}]\leq -\frac{2}{3}+\epsilon.$$ Taking limit we obtain $$\zeta^{-}_{2}\geq \frac{2}{3}.$$
\bigskip
Proof of the second part is similar.
\end{proof}

\bibliographystyle{plain}
\addcontentsline{toc}{section}{\bibname} 

\bibliography{references}

\end{document}